\documentclass[11pt]{article}
\usepackage[utf8]{inputenc}
\usepackage{lmodern}
\usepackage[normalem]{ulem} 
\usepackage{float}
\usepackage[margin=1in]{geometry}
\usepackage{enumitem}
\usepackage[utf8]{inputenc} 
\usepackage{tocloft}
\usepackage{graphicx, titlesec}
\graphicspath{/figures}
\usepackage{amscd,amsmath,amsthm,amssymb}
\usepackage{verbatim}
\usepackage[all]{xy}
\usepackage{color}
\usepackage{url}
\usepackage[colorlinks=true, linkcolor=purple, citecolor=purple, urlcolor=purple]{hyperref}
\usepackage{booktabs}
\usepackage{tikz-cd}
\usetikzlibrary{patterns}
\usetikzlibrary{calc}
\usepackage{ytableau}
\usepackage[norelsize,lined,ruled, linesnumbered ]{algorithm2e} 
\usepackage{booktabs}
\usepackage{cleveref}

\usepackage{mathtools}

\usepackage{longtable}
\usepackage{nicematrix}
\usepackage{multirow}
\usepackage{pgf}
\usepackage{subcaption}

\newcommand\independent{\protect\mathpalette{\protect\independenT}{\perp}}
\def\independenT#1#2{\mathrel{\rlap{$#1#2$}\mkern2mu{#1#2}}}

\newcommand{\RR}{\mathbb{R}}
\newcommand{\CC}{\mathbb{C}}

\definecolor{ao}{rgb}{0.0, 0.5, 0.0}
\definecolor{myred}{rgb}{0.81, 0.09, 0.13}

\newtheorem{theorem}{Theorem}[section]

\theoremstyle{definition}
\newtheorem{definition}[theorem]{Definition}

\newtheorem{example}[theorem]{Example}
\newtheorem{remark}[theorem]{Remark}

\theoremstyle{plain}
\newtheorem{lemma}[theorem]{Lemma}
\newtheorem{proposition}[theorem]{Proposition}
\newtheorem{corollary}[theorem]{Corollary}

\title{Decomposing conditional independence ideals with hidden variables}
\author{Yulia Alexandr, Kristen Dawson, Hannah Friedman, Fatemeh Mohammadi,\\ Pardis Semnani, and Teresa Yu}
\date{}

\begin{document}

\maketitle

\begin{abstract}
 \noindent We study a family of determinantal ideals whose decompositions encode the structural zeros in conditional independence models with hidden variables. 
  We provide explicit decompositions of these ideals and,  for certain subclasses of models, we show that this is a decomposition into radical ideals by 
  displaying Gr\"obner bases
  for the components. We 
 identify conditions under which
   the components
  are prime, and establish formulas for the dimensions of these prime 
  ideals.
 We show that the components in the decomposition can be 
grouped
 into equivalence classes defined by their combinatorial structure, 
 and we derive a closed formula for the number of such~classes.
\end{abstract}

\renewcommand{\thefootnote}{}
\footnotetext{\noindent{\!\!\!\!\!\!\!\!\!\!\!\small{Keywords:}} Conditional independence models; Determinantal ideals; Gr\"obner bases; Prime decomposition. 
}

{\hypersetup{linkcolor=black}}

\section{Introduction}
\subsection{Motivation}
Conditional independence is a key notion in statistical modeling~\cite{studeny2006probabilistic}, offering a structural interpretation for Markov fields and graphical models~\cite{maathuis2018handbook}. Conditional independence (CI) models have been extensively studied in algebraic statistics~\cite{drton2008lectures, sullivant2023algebraic} as well as in combinatorics~\cite{andersson1993lattice, caines2022lattice}.
One of the main questions in this context is the characterization of probability distributions that satisfy a given collection of CI statements. In a more general framework, some of the random variables in a CI model may be unobserved (or hidden), while others are observed. In this setting, the main question becomes whether certain dependencies among the observed variables arise from constraints among the hidden variables~\cite{Steudel-Ay}. This problem has an algebraic analogue in terms of the properties of the associated CI ideal~\cite{drton2008lectures, sullivant2023algebraic}. More precisely, the decomposition of this ideal leads to inferring additional (in)dependencies among random variables. Such inferences are known as CI properties~\cite{pearl1989conditional}, and our focus will be on the intersection property.
\medskip

The joint distribution of $Y, Y_1, Y_2, H_1,H_2$  is said to satisfy the generalized {\it intersection property}~if
\begin{eqnarray}\label{eq:intersection_property}
Y \independent Y_1 \mid \{Y_2, H_1\} \quad \text{and} \quad Y \independent Y_2 \mid \{Y_1, H_2\} \quad \Rightarrow \quad Y \independent \{Y_1, Y_2\}\mid \{H_1,H_2\}.
\end{eqnarray}
In the classical setting where $H_1 = H_2$ and these variables are observed, the intersection property is known to hold under the assumption of strictly positive densities~\cite{pearl2009causality}. A complete characterization of the necessary and sufficient conditions for the validity of the intersection property in the case of distributions with a continuous density is given in~\cite{peters2015intersection}. However, corresponding results remain unknown when hidden variables are present.

In the absence of the variables $H_1$ and $H_2$, the corresponding CI ideals are generated by binomials, and such ideals (together with their decompositions) have been extensively studied~\cite{fink2011binomial,herzog2010binomial,rauh2013generalized,swanson2013minimal}. The intersection property without $H_1$ and $H_2$ is investigated in~\cite{fink2011binomial,herzog2010binomial}, while~\cite{swanson2013minimal} analyzes the case $H_1 = H_2$ with both variables observed. This setup also gives rise to binomial ideals and leads to the inference $X \independent \{Y_1, Y_2\} \mid H_1$.  
By contrast, when $H_1$ and $H_2$ are hidden variables, the defining polynomials of the CI ideal may involve an arbitrary number of terms and reach arbitrarily high degrees, resulting in significant computational challenges~\cite{clarke2020conditional,pfister2021primary}. Although these ideals, as well as those arising in their decompositions, are determinantal, they are generated by higher minors.

In this paper, we focus on the case where $H_1$ and $H_2$ are hidden variables with state spaces of sizes $1$ and $t-1$, respectively. In this setting, the intersection property~\eqref{eq:intersection_property} simplifies to
\begin{equation}\label{eq:intersection_property_simplified}
Y \independent Y_1 \mid Y_2 
\quad \text{and} \quad 
Y \independent Y_2 \mid \{Y_1, H\} 
\quad \Rightarrow \quad 
Y \independent \{Y_1, Y_2\} \mid H.
\end{equation}
where $H:=H_2$ in \eqref{eq:intersection_property}. We study the ideals associated with the CI family 
\begin{equation}\label{C}
C = \big\{\, Y \independent Y_1 \mid Y_2, \;\; 
                     Y \independent Y_2 \mid \{Y_1, H\} \,\big\}
\end{equation}
and their primary decompositions. In our setting, the original CI ideals are generated by minors of two different sizes, namely $2$-minors and $t$-minors, where the state space of $H$ is of size $t-1$. We show that the components appearing in their decompositions may be generated by $1$-minors, $2$-minors, $(t-1)$-minors, and $t$-minors. These ideals can be described using combinatorial structures such as grids and hypergraphs, which provide a framework for analyzing their properties. This construction yields a new class of prime mixed determinantal ideals whose Gröbner bases consist precisely of the generating minors. For related but distinct families of such ideals, see~\cite{HS04}.

\subsection{Basic notions and main results} 
We consider three observed random variables $Y$, $Y_1$, and $Y_2$ with finite state spaces of sizes $d$, $k_1$, and $k_2$, respectively. Additionally, we introduce one hidden variable $H$ with a state space of size $t-1$. 
 Throughout this paper, we will assume that $2\leq  t \leq k_2$ and $2\leq k_1$.  The joint distribution of the observed variables $Y, Y_1,$ and $Y_2$ can be expressed by the non-negative tensor $P=(p_{ij\ell})\in\RR^{d\times k_1 \times k_2}$.

To encode the space of such distributions satisfying CI statements, we consider the following setup. Let $\widetilde{X}=(x_{ij\ell})\in\CC^d\otimes\CC^{k_1}\otimes\CC^{k_2}$ denote a $3$-way tensor with indeterminate entries (see Figures~\ref{fig: tensor-matrix-grid: tensor1} and~\ref{fig: tensor-matrix-grid: tensor}), and let $X = (x_{i(j,\ell)})\in\CC^{d\times k_1 k_2}$ denote its flattening to a matrix whose columns are indexed by the Cartesian product $[k_1]\times[k_2]$ (see Figure~\ref{fig: tensor-matrix-grid: matrix}). The entries $\{x_{i(j,\ell)}\}$ represent an unknown probability distribution $p_{ij\ell}$ on the observed variables. Let $\CC[X]=\CC[x_{i(j,\ell)}:i\in[d],(j,\ell)\in[k_1]\times[k_2]]$ denote the polynomial ring with variables given by the entries of $X$.

We let $R_i$ denote the subset $\{i\}\times[k_2]\subset [k_1]\times[k_2]$; given Figure~\ref{fig: tensor-matrix-grid: grid}, we refer to $R_i$ as a \emph{row slice} of $[k_1]\times[k_2]$. Similarly, we let $C_j$ denote the subset $[k_1]\times\{j\}\subset[k_1]\times[k_2]$, and refer to it as a \emph{column slice} of $[k_1]\times[k_2]$.
For any subset $U\subseteq [k_1]\times [k_2]$, let $X_U$ denote the submatrix of $X$, induced by the columns of $X$ indexed by~$U$. Consider the collection of \textit{conditional independence (CI) statements} $C=\{Y\independent Y_1
\mid \{Y_2\}, Y\independent Y_2
\mid \{Y_1, H\}\}$. The ideals corresponding to these statements~are
\begin{align*}
\mathcal I_{Y\independent Y_1
\mid\{Y_2\}} &=\langle 2\text{-minors of } X_{C_j} \text{ for all }j\in [k_2]\rangle \subseteq \CC[X],\\
 \mathcal I_{Y\independent Y_2\mid\{Y_1, H\}} &=\langle t\text{-minors of } X_{R_i} \text{ for all }i\in [k_1]\rangle \subseteq \CC[X],
\end{align*}
respectively. The \textit{conditional independence (CI) ideal} associated to $C$ is defined as
$$\mathcal I_{C} =\mathcal I_{Y\independent Y_1\mid \{Y_2\}} + \mathcal I_{Y\independent Y_2\mid \{Y_1, H\}}\subseteq \CC[X].$$
If $Y,Y_1,Y_2,H$ are random variables that satisfy the CI statements in $C$, then the joint distribution $P=(p_{ij_1j_2})$ of the observed variables $Y,Y_1,Y_2$ lies in the vanishing set of $\mathcal{I}_C$. 
We are interested in the prime decomposition of the CI ideal $\mathcal{I}_C$ and the statistical properties captured by its minimal primes.
\begin{figure}[h]
    \centering
    \begin{subfigure}[t]{\textwidth}
    \centering
        \begin{tikzpicture}[scale=0.8]

            \def\slices{4}
            \def\rows{2}
            \def\cols{3}
            \def\shifth{-1}  
            \def\shiftv{0.5}  
            \def\cellsize{1.5}
            
            \pgfmathsetmacro{\a}{\cellsize*\rows}
            \pgfmathsetmacro{\b}{\cellsize*\cols}
            \pgfmathsetmacro{\sx}{\shifth*\slices}
            \pgfmathsetmacro{\sy}{\shiftv*\slices}

            \coordinate (A) at (0,0);
            \coordinate (B) at (\b,0);
            \coordinate (C) at (\b,-\a);
            \coordinate (D) at (0,-\a);
            
            \coordinate (A2) at (\sx,\sy);
            \coordinate (B2) at (\b+\sx,\sy);
            \coordinate (C2) at (\b+\sx,-\a+\sy);
            \coordinate (D2) at (\sx,-\a+\sy);

            \draw[thick] (A) -- (B) -- (C) -- (D) -- cycle ;      
            \draw[thick] (D2) -- (A2) -- (B2);   
            \draw[thick] (A) -- (A2);
            \draw[thick] (B) -- (B2);
            \draw[thick] (D) -- (D2);
            

            \foreach \j in {1,...,\cols} {
            \foreach \i in {1} {
            \draw[fill = cyan, opacity = 0.7] (\j*\cellsize - \cellsize, -\i*\cellsize) rectangle (\j*\cellsize, -\i*\cellsize + \cellsize);
            }
            }

            \foreach \i in {1,...,\slices} {
            \foreach \j in {1} {
            \draw[fill = cyan, opacity = 0.6] (A) -- (A2) -- (\sx,-\a+\sy+\cellsize) -- (0,-\a+\cellsize) -- cycle;
            }
            }
            
            \draw[fill = cyan, opacity = 0.7] (A2) -- (B2) -- (B) -- (A) -- cycle;

            \draw[fill = yellow, opacity=0.6] (\cellsize,0) rectangle (2*\cellsize, - \cellsize*\rows);
            
            \draw[fill = yellow, opacity=0.6] (\cellsize,0) -- (2*\cellsize,0) -- (\sx + 2*\cellsize,\sy) -- (\sx + \cellsize,\sy) -- cycle;

            \foreach \i in {1,...,\slices} {
            \draw (\i*\shifth, -\a+\i*\shiftv) -- (\i*\shifth, \i*\shiftv) -- (\b+\shifth*\i, \i*\shiftv);
            }
            
            \foreach \i in {1,...,\slices} {
            \draw (\i*\shifth, -\a+\i*\shiftv) -- (\i*\shifth, \i*\shiftv) -- (\b+\shifth*\i, \i*\shiftv);
            }
            
            \foreach \i in {1,...,\rows} {
            \draw (\sx,\sy - \i*\cellsize) -- (0,-\i*\cellsize) -- (\b, -\i*\cellsize);
            }
            
            \foreach \i in {1,...,\cols} {
            \draw (\sx + \i*\cellsize,\sy) -- (\i*\cellsize,0) -- (\i*\cellsize, -\a);
            }

            \foreach \i in {1,...,\rows} {
            \foreach \j in {1,...,\cols} {
            \node at (\j*\cellsize - \cellsize/2, -\i*\cellsize + \cellsize/2) {$x_{4\i\j}$};
            }
            }

            \foreach \i in {1,...,\slices} {
            \foreach \j in {1,...,\rows} {
            \node at (\sx-\i*\shifth + \shifth/2,   \sy - \i*\shiftv -\j*\cellsize + \cellsize/2) {$x_{\i\j1}$};
            }
            }
            
            \foreach \i in {1,...,\slices} {
            \foreach \j in {1,...,\cols} {
            \node at (\sx + \cellsize*\j -\i*\shifth - 4*\cellsize/5,   \sy - \i*\shiftv +\shiftv/2) {$x_{\i1\j}$};
            }
            }
            
            \end{tikzpicture}
            \caption{The $d \times k_1 \times k_2$ tensor $\widetilde{X}$ for the values $d=4, k_1=2, k_2=3$.}
            \label{fig: tensor-matrix-grid: tensor1}
    \end{subfigure}\\
    ~
    \begin{subfigure}[t]{\textwidth}
    \vspace{5mm}
        \centering
        \begin{tikzpicture}[scale=0.6]

        \def\rows{2} 
        \def\cols{3} 
        \def\slices{4} 
        \def\cellsize{1.5}
        \pgfmathsetmacro\shifth{\cellsize*\cols + 0.2} 
        \def\shiftv{0.6} 
        
        \foreach \i in {1,...,\slices} {
          \begin{scope}[shift={(\shifth*\i,-\shiftv*\i)}]
            \draw[thick] (0,0) rectangle (\cellsize*\cols,-\cellsize*\rows);
            \foreach \j in {1,...,\rows} {
                \foreach \l in {1,...,\cols} {
                \ifnum\j=1\relax 
                \draw[fill=cyan, opacity = 0.7] (\cellsize*\l-\cellsize,-\cellsize*\j+\cellsize) rectangle (\cellsize*\l,-\cellsize*\j);
                \fi
                \ifnum \l = 2\relax
                \draw[fill=yellow, opacity=0.6] (\cellsize*\l-\cellsize,-\cellsize*\j+\cellsize) rectangle (\cellsize*\l,-\cellsize*\j);
                \else 
                \draw (\cellsize*\l-\cellsize,-\cellsize*\j+\cellsize) rectangle (\cellsize*\l,-\cellsize*\j);
                \fi
                \node at (\cellsize*\l-\cellsize/2,-\cellsize*\j+\cellsize/2) {$x_{\i\j\l}$};
                }
            }
          \end{scope}
        }

        \pgfmathsetmacro{\totx}{\shifth*\slices}
        \pgfmathsetmacro{\toty}{\shiftv*\slices}
        \foreach \x/\y in {0/0,\cellsize*\cols/0,0/-\cellsize*\rows,\cellsize*\cols/-\cellsize*\rows} {
              \pgfmathsetmacro{\xstart}{\x+\shifth}
              \pgfmathsetmacro{\ystart}{\y-\shiftv}
              \pgfmathsetmacro{\xend}{\x+\totx}
              \pgfmathsetmacro{\yend}{\y-\toty}
              \draw[dashed] (\xstart,\ystart) -- (\xend,\yend);
        }
        
        \end{tikzpicture}
        \caption{The $d$ slices of tensor $\widetilde{X}$, each of size $k_1 \times k_2$.}
    \label{fig: tensor-matrix-grid: tensor}
    \end{subfigure}\\
    ~ 
    \begin{subfigure}[b]{0.5\textwidth}
    \vspace{5mm}
    \centering
    \begin{tikzpicture}[scale = 0.6]
        \def\rows{2} 
        \def\cols{3} 
        \def\slices{4} 
        \def\cellsize{1.8}
        \pgfmathsetmacro\shifth{\cellsize*\cols + 0.2} 
        \def\shiftv{0.6} 
        
          \begin{scope}[shift={(0,0)}]
            \draw[thick] (0,0) rectangle (\cellsize*\cols*\rows,-\cellsize*\slices);
            \foreach \i in {1,...,\slices}{
                \foreach \l in {1,...,\cols} {
                    \foreach \j in {1,...,\rows} {
                      \ifnum\j=1\relax
                      \draw[fill = cyan, opacity=0.7] (\cellsize*\l*\rows - \cellsize*\rows + \cellsize*\j-\cellsize,-\cellsize*\i+\cellsize) rectangle (\cellsize*\l*\rows - \cellsize*\rows + \cellsize*\j,-\cellsize*\i);
                      \fi
                      \ifnum \l = 2\relax
                      \draw[fill = yellow, opacity=0.6] (\cellsize*\l*\rows - \cellsize*\rows + \cellsize*\j-\cellsize,-\cellsize*\i+\cellsize) rectangle (\cellsize*\l*\rows - \cellsize*\rows + \cellsize*\j,-\cellsize*\i);
                      \else
                      \draw (\cellsize*\l*\rows - \cellsize*\rows + \cellsize*\j-\cellsize,-\cellsize*\i+\cellsize) rectangle (\cellsize*\l*\rows - \cellsize*\rows + \cellsize*\j,-\cellsize*\i);
                      \fi
                       \node at (\cellsize*\l*\rows - \cellsize*\rows + \cellsize*\j-\cellsize/2,-\cellsize*\i+\cellsize/2) {$x_{\i(\j,\l)}$};
                    }
                }
            }
          \end{scope}
    \end{tikzpicture}
    \caption{The $d\times k_1 k_2$ matrix $X$, a flattening of $\widetilde{X}$.}
    \label{fig: tensor-matrix-grid: matrix}
    \end{subfigure}%
    ~
    \begin{subfigure}[b]{0.5\textwidth}
            \vspace{5mm}
            \centering
            \begin{tikzpicture}[scale=0.8]
            \def\rows{2} 
            \def\cols{3} 
            \def\cellsize{1.5}
            \tikzset{
                dot/.style={circle, fill, inner sep=0pt, minimum size=6pt},
                circled/.style={circle, draw, inner sep=4pt}
            }
            \begin{scope}]
            \foreach \x in {0,...,2}{
                \node at (\cellsize*\x,1.9) {}; 
                \pgfmathtruncatemacro{\label}{int(\x + 1)};
                \node[label, font=\small] at (\cellsize*\x,\cellsize*1+0.6) {\label};
            }
            \node at (-0.8,\cellsize*1) {1};
            \node at (-0.8,0) {2};
        

            \draw[fill = cyan, opacity = 0.7] (-0.2,\cellsize + 0.2) rectangle (\cellsize*2 + 0.2,\cellsize-0.2);
            \draw[fill = yellow, opacity = 0.6] (-0.2 + \cellsize,\cellsize + 0.3) rectangle (\cellsize + 0.2,-0.3);

                \foreach \x in {0,...,2}{

                \node[dot] at (\cellsize*\x,\cellsize*1) {};
                \node[dot] at (\cellsize*\x,0) {};
                
            }

            \end{scope}
            \end{tikzpicture}
            \caption{The labeled grid $[k_1] \times [k_2]$.}
            \label{fig: tensor-matrix-grid: grid}
        \end{subfigure}

\caption{In the figures, $d=4$, $k_1 = 2$ and $k_2 = 3$. The rows and columns of the grid in (d) correspond to the rows and columns of each slice of $\widetilde{X}$ in (b) respectively. For $i=1$, the blue area indicates (a) $\widetilde{X}_{R_i}$, (b) $\widetilde{X}_{R_i}$, (c) $X_{R_i}$, and (d) $R_i$. For $j=2$, the yellow area indicates (a) $\widetilde{X}_{C_j}$, (b) $\widetilde{X}_{C_j}$, (c) $X_{C_j}$, and (d) $C_j$. Green areas belong to both the blue and yellow areas.}
\label{fig: tensor-matrix-grid}
\end{figure}

There is a rich combinatorial structure associated with the CI ideal $\mathcal{I}_C$, arising from its interpretation as a determinantal edge ideal. A \emph{hypergraph} $H$ with the vertex set $[k_1] \times [k_2]$ is a collection of subsets of $[k_1] \times [k_2]$. For subsets $A \subseteq [d]$ and $B \subseteq [k_1] \times [k_2]$ of the same size, we denote by $[A\mid B]$ the minor of $X$ formed by selecting rows indexed by $A$ and columns indexed by $B$. The \textit{hypergraph ideal} associated to $H$ is defined as
\begin{eqnarray}\label{eq:ideal_hypergraph}
    I(H):=\{[A\mid B] : A\subseteq[d], B\in H, |A| = |B|\}\subseteq \CC[X].
\end{eqnarray}

\noindent From this definition, we see that the ideal $\mathcal{I}_C$ we defined earlier is a hypergraph ideal $I(\Delta)$, where $$\Delta = \Delta(k_1, k_2, t, 2) := \left\{\binom{R_i}{t}, \binom{C_j}{2} : i\in[k_1], j\in[k_2]\right\}.$$
Here, $\binom{R_i}{t}$ and $\binom{C_j}{2}$ denote the collections of all $t$-element subsets of $R_i$ and all $2$-element subsets of $C_j$, respectively, for all $i\in[k_1]$ and $j\in[k_2]$.

A \emph{zero set} $S \subseteq [k_1] \times [k_2]$ determines  entries in the tensor $P$ which we set to zero, indicating those states of $Y_1 \times Y_2$ that have zero probability. 
Our main result is a decomposition of the ideal $\mathcal{I}_C$ into a family of ideals $I_S$, each associated with a zero set $S \subseteq [k_1] \times [k_2]$, and defined explicitly in Definition~\ref{def:IS-for-k_1=2}. The precise statement is as follows.
\begin{theorem}\label{thm:main-result}
Let $2 \leq k_1$ and $2 \leq t \leq \min\{k_2,d\}$, with $d$, $k_1$, and $k_2$ arbitrary. Then  
$\sqrt{\mathcal{I}_C} = \bigcap_{S \subseteq [k_1] \times [k_2]} \sqrt{I_S}$.  
In the special case where $k_1 = 2$, 
the minimal components in this decomposition are radical, i.e.,  
$\sqrt{\mathcal{I}_C} = \bigcap_{ S\subseteq [k_1] \times [k_2]} I_S$.  
Moreover, if $d = t$, then the minimal components are~prime.
\end{theorem}

Beyond its algebraic significance, Theorem~\ref{thm:main-result} 
also admits a natural statistical interpretation.

\begin{corollary}\label{cor:Iempty int property}
Consider the CI model $C$ from~\eqref{C}, where $Y$, $Y_1$, and $Y_2$ are observed random variables with finite state spaces of sizes $d$, $k_1$, and $k_2$, respectively, and $H$ is a hidden variable with state space of size $t-1$. Assume that $d$, $k_1$, $k_2$, and $t$ are arbitrary integers such that $2 \leq t \leq \min\{k_2, d\}$. Then the intersection property in~\eqref{eq:intersection_property_simplified} holds for these variables.  
More precisely, the component $I_\emptyset$ corresponding to $S = \emptyset$ in the decomposition of $\mathcal{I}_C$ characterizes joint distributions with full support, i.e., those without structural zeros in the probability table. This component contains the polynomials associated with the conditional independence relation 
$Y \independent \{Y_1, Y_2\} \mid H$.
\end{corollary}

The intersection axiom does not hold for general subsets $S\subseteq [k_1] \times [k_2]$ in the presence of structural zeros; see Remark~\ref{I_S:intersection}.

\subsection{Prior related works}

Before presenting the main content, we briefly review some related work. The ideals $\mathcal{I}_C$ with hidden random variables have been studied under several specific parameter choices: for $t = k_2$ in \cite{clarke2020conditional}, for $t = 3$ in \cite{clarke2022conditional}, and for $t = k_2$ and $t = 2$ in \cite{clarke2023matroid}.
In all these cases, the corresponding ideals $I_S$ have generators of only three types: $1$-minors,  $2$-minors and $t$-minors. In contrast, in our setting the ideals $I_S$ involve four types of generators, $1$-minors, $2$-minors, $(t-1)$-minors, and $t$-minors, making the structure more intricate.

\subsection{Outline}  
In Section~\ref{sec:decomposition}, we define a new class of determinantal ideals $I_S$ depending on a subset $S \subseteq [k_1] \times [k_2]$, which represents the set of structural zeroes.  
We prove that the conditional independence ideal $\mathcal \mathcal{I}_C$ decomposes as the intersection of these ideals $I_S$. In Sections~\ref{sec:minimality}, \ref{sec:primeness}, and \ref{sec:dim}, we analyze this decomposition with a focus on the case $k_1 = 2$.
In Section~\ref{sec:minimality}, we characterize the subsets $S \subseteq [k_1] \times [k_2]$ required for a minimal decomposition and present a closed formula for the number of such ideals, up to equivalence of combinatorial types.  
In Section~\ref{sec:primeness}, we show that the natural generating set for the ideals $I_S$, consisting of certain mixed minors, forms a Gr\"{o}bner basis, and conclude that these ideals are radical. We also prove that these ideals are prime in the case where $d = t$.  
In Section~\ref{sec:dim}, we study the dimension of the minimal prime ideals $I_S$ and provide exact formulas.

\medskip
Table~\ref{tab:notation} summarizes the key notation used in the paper. Some notation is introduced in later sections. Any notation not explicitly defined in the main text can be found in this table.
\begin{longtable}{cl}
\caption{Key notation.} \label{tab:notation} \\
\toprule
Symbol & Description\\
\midrule
\endfirsthead

\toprule
Symbol & Description\\
\midrule
\endhead

$[n]$ & The set $\{1,2,\ldots,n\}$\\
$\mathcal{I}_C$ & The original conditional independence ideal\\
$H(S)$ & Hypergraph associated to $S$ \\
$\overline{H}$ & Closure of the hypergraph $H$ \\
$I(H)$ & Ideal associated to hypergraph $H$ \\
$I_{S}$ & Ideal associated to $\overline{H(S)}$ \\
$R_i$ & Row slice $\{i\} \times [k_2]$ of $[k_1] \times [k_2]$\\
$C_j$ & Column slice $[k_1] \times \{j\}$ of $[k_1] \times [k_2]$\\
$Z(r,S)$ & Indices $i \in [k_2]$ such that $(r,i) \in S$\\
$NZ(r,S)$ & Indices $i \in [k_2]$ such that $(r,i) \notin S$\\
$\mathcal{C}(S)$ & Union of sets $[k_1] \times \{i\}$ such that $([k_1] \times \{i\})\cap S=\emptyset$ \\
$(u,v)$ & Combinatorial type of a subset $S$\\
$X$ & The $d\times k_1k_2$ matrix of indeterminates\\
$A_U$ & Submatrix of a matrix $A$ on columns labeled by $U$\\
$F_S$ & The set of minor generators of $I_S$ corresponding to $S$ \\ 
$F_t(X)$ & The set of all $t$-minors of $X$\\
$I_t(X)$ & Ideal generated by $F_t(X)$\\
\bottomrule
\end{longtable}

\section{Decomposition theorem}\label{sec:decomposition}

In this section, we define  hypergraphs associated to the sets $S \subseteq [k_1] \times [k_2]$ of structural zeros. 
We then give the decomposition of the CI ideal $\mathcal{I}_C$ in terms of the ideals of these hypergraphs.

\begin{definition}\label{def: step closure}
 Given a hypergraph $H$, we define its \emph{closure}, denoted by $\overline{H}$, as follows. Set $H_0 := H$, and for each integer $q \geq 1$, define
$$H_q := H_{q-1} \cup \{\{i_1,\ldots,i_r,j\} : \text{ for any }\{i,j\}, \{i_1,\ldots,i_r,i\}\in H_{q-1}\} \text{ for all $q\geq 1$}.$$
Let $n\in\mathbb{N}$ be the index such that $H_{n+1}=H_{n}$. We then define the closure of $H$ as $\overline{H} := H_n$.
We~let 
\[Z(r,S) := \{\, i \in [k_2] : (r,i) \in S \,\}
\quad\text{and}\quad 
NZ(r,S) := \{\, i \in [k_2] : (r,i) \notin S \,\},\]
where $Z(r,S)$ denotes the set of column indices $i$ such that the entry in row $r$ and column $i$ lies in the marked zero set $S$, and $NZ(r,S)$ denotes the complementary set of column indices for which this is not the case.
\end{definition}

\begin{definition}\label{def:IS-for-k_1=2}
 
Let $2 \leq  k_1$ and $2\leq t \leq \min\{k_2,d\}$.
Given a subset $S \subseteq [k_1] \times [k_2]$, we let $H(S)$ be the hypergraph on $[k_1] \times [k_2]$ with the following hyperedges:

\begin{enumerate}
    \item $1$-subsets of $S$; \label{def: step loops} 
       \item $2$-subsets $\{(i_1,j),\, (i_2, j)\}$ not intersecting $S$ for $j\in [k_2]$;
       \label{def: step 2-minors} 
    \item $(t-1)$-subsets of $\{i,j\} \times (NZ(i,S) \cap NZ(j,S))$ with $|Z(i,S) \backslash Z(j,S)|, \,|Z(j,S) \backslash Z(i,S)| \geq 1$;\label{def: step t-1 minors} 
    \item $t$-subsets $\{(i, j_1), \ldots, (i, j_t)\}$ not intersecting $S$. \label{def: step t minors} 

\end{enumerate}

We let $\overline{H(S)}$ be the closure of the hypergraph $H(S)$ and define $I_S$ to be the associated hypergraph ideal of  $\overline{H(S)}$ as defined in \eqref{eq:ideal_hypergraph}, i.e., $I_S\coloneqq I(\overline{H(S)})$.
\end{definition}

The ideals $I_S \subseteq \mathbb{C}[X]$ are defined by rank conditions on certain submatrices of $X$, so they are determinantal ideals. We can write them using the following notation in the case of $k_1 =2$.

For any submatrix $Y$ of $X$, let $F_t(Y)$ denote the set of $t$-minors of $Y$, and let $I_t(Y)$ represent the ideal in the ring $\mathbb{C}[X]$ generated by $F_t(Y)$. Recall that $X_U$ denotes the submatrix of $X$ indexed by the columns in the set $U \subseteq [k_1] \times [k_2]$.

\begin{definition}\label{def:k_1=2}
Let $k_1 = 2$ and $2\leq t \leq \min\{k_2,d\}$.
Let $\mathcal{C}(S)$ denote the union of all column slices $C_i$ not intersecting the zero set, that is, $\mathcal{C}(S) = \bigcup_{(1,i),(2,i) \notin S} C_i$.
The ideals $I_\emptyset$ and $I_S$ are generated by the sets $F_\emptyset$ and $F_S$, respectively; these sets are defined as
\begin{align*}
    F_\emptyset &\coloneqq \bigcup_{i \in [k_2]} F_2(X_{C_i}) \cup F_t(X),\\
    F_S & \coloneqq F_1(X_S) \cup \!\!\!\!\!\!\!\!
    \bigcup_{\substack{i\in\{1,2\}\\ (1,i),(2,i) \notin S}}
    \!\!\!\!\!\!\!\!
    F_2(X_{C_i}) \cup F_{t-1}(X_{\mathcal C(S)}) \cup F_t(X_{((R_1) \setminus S) \cup \mathcal C(S)})
    \cup 
    F_t(X_{((R_2)  \setminus S) \cup \mathcal C(S)}),
\end{align*}
where the set of $(t\!-\!1)$-minors $F_{t-1}(X_{\mathcal{C}(S)})$ is contained in $F_S$ only if, 
for each $i \in \{1,2\}$, there exists a $j \in [k_2]$ such that $(i,j) \in S$ and 
$([2]\setminus \{i\}, j) \notin S$; that is, 
$|Z(1,S) \setminus Z(2,S)| \ge 1$ and $|Z(2,S) \setminus Z(1,S)| \ge 1$.
\end{definition}

Note that in $F_S$, we may include in our generating set $t$-minors of $X_{(R_1 \setminus S) \cup \mathcal C(S)}$ or $t$-minors of $X_{R_1 \cup \mathcal C(S)}$. The additional $t$-minors in $F_t(X_{R_1 \cup \mathcal C(S)}) \setminus F_t(X_{(R_1\setminus S) \cup \mathcal C(S)})$ are contained in the ideal $I_1(X_S)$; see Remark \ref{remark: some minors are redundant}.

\begin{remark}\label{I_S:intersection}
The set of generators $F_\emptyset$ of $I_\emptyset$ contains all $t$-minors of $X$, which correspond to the conditional independence statement $Y \independent \{Y_1, Y_2\} \mid H$. Such a statement requires that all $t$-minors of $X$ lie in the ideal, which is not necessarily the case for arbitrary $S$. 

\end{remark}

\begin{example}
Let $k_1 = 2$, $k_2 = 5$, and $t = 4$. Fix $S = \{(1,1), (2,2)\}$, as in Figure~\ref{fig:k_1=2-hypergraph-example}. 
\begin{figure}[h]
       \centering
    \begin{tikzpicture}[scale=0.8]
    \tikzset{
        dot/.style={circle, fill, inner sep=0pt, minimum size=6pt},
        circled/.style={circle, draw, inner sep=4pt}
    }
    \foreach \x in {0,...,4}{
        \node at (\x,1.9) {}; 
        \pgfmathtruncatemacro{\label}{int(\x + 1)};
        \node[label, font=\small] at (\x,1.6) {\label};
    }
    \node at (-0.8,1) {1};
    \node at (-0.8,0) {2};

    \foreach \x in {0,...,4}{
        \node[dot] at (\x,1) {};
        \node[dot] at (\x,0) {};
    }
    \node[circled] at (0,1) {};
    \node[circled] at (1,0) {};

\end{tikzpicture}
\caption{The labeled grid $[2] \times [5]$ for $S=\{(1,1), (2,2)\}$.}
\label{fig:k_1=2-hypergraph-example}
\end{figure}
The hypergraph $H(S)$ contains the following hyperedges:
\begin{enumerate}
    \item[1.  ] $\{(1,1)\}, \{(2,2)\}$ from Step~\ref{def: step loops};
    \item[2.  ] $\{(1,3), (2,3)\}, \{(1,4), (2,4)\}, \{(1,5), (2,5)\}$ from Step~\ref{def: step 2-minors};
    \item[3.  ] all 3-subsets of $\{1,2\} \times \{3,4,5\}$ from Step~\ref{def: step t-1 minors};
    \item[4.  ] $\{1\} \times \{2,3,4,5\}$ and $\{2\} \times \{1,3,4,5\}$ from Step~\ref{def: step t minors}.
\end{enumerate}
Taking the closure, we additionally obtain all $4$-subsets of $\{(1,2)\} \cup (\{1,2\} \times \{3,4,5\})$ and 
$\{(2,1)\} \cup (\{1,2\} \times \{3,4,5\})$ as new hyperedges in $\overline{H(S)}$. Computation in 
\texttt{Macaulay2}~\cite{M2} confirms that, when $d=4$, the following 
associated ideal is prime:
\begin{align*}
    I_S &= I_1(X_{\{(1,1), (2,2)\}}) + \sum_{i \in \{3, 4, 5\}} I_2(X_{C_i}) + I_{3}(X_{\{1,2\} \times \{3,4,5\}}) \\
    &+ I_4(X_{\{(1,2)\} \cup (\{1,2\} \times \{3,4,5\})})
    +
    I_4(X_{\{(2,1)\} \cup (\{1,2\} \times \{3,4,5\})}).
\end{align*}
\end{example}

\begin{remark} \label{remark: some minors are redundant}
   Before presenting the main theorem, we recall some properties of the ideal $I_S$ that will be important for its proof.  For any $2$-subset $\{(i_1,j_1), (i_2, j_2)\} \in H(S)$, we say that $(i_1, j_1)$ is \textit{identified} with $(i_2, j_2)$ via $S$, reflecting the fact that the ideal $I_S$ contains the corresponding $2$-minors. As a consequence, for any matrix $A$ in the variety associated with $I_S$, denoted $V(I_S)$, the columns $A_{(i_1,j_1)}$ and $A_{(i_2, j_2)}$ are either zero or scalar multiples of each other.

In Steps~\ref{def: step 2-minors} and~\ref{def: step t minors} of Definition~\ref{def:IS-for-k_1=2}, we consider only those hyperedges that do not contain elements of~$S$, as Step~\ref{def: step loops} guarantees that $x_{i(j_1,j_2)} \in I_S$ for all $i \in [d]$ and $(j_1, j_2) \in S$. Therefore, any minor of size at least 2 containing $x_{i(j_1, j_2)}$ will be redundant in the generating set of $I_S$.
\end{remark}

In general, we obtain the following decomposition theorem. 
The minimal components of this decomposition are always radical for $k_1 = 2$ and prime for $k_1 = 2$ and $d = t$, as shown in Section~\ref{sec:primeness}. However, for $k_1 \geq 3$, they may not be prime (see Example~\ref{example: I_S is not prime if k>=2}).

\begin{theorem} \label{thm: decomposition theorem}
    In the case of $2 \leq k_1$ and $2\leq t \leq \min\{k_2, d\}$, we have 
    \begin{align} \label{eq: decomposition}
        \sqrt{\mathcal{I}_C} = \bigcap_{S\subseteq [k_1]\times [k_2]} \sqrt{I_S}.
    \end{align}
\end{theorem}

\begin{proof}
To prove the theorem, we consider the associated varieties of the ideals in \eqref{eq: decomposition} and show that $V(\mathcal{I}_C)=\bigcup_{S\subseteq [k_1]\times[k_2]} V(I_S)$. First, let $A\in V(I_S)$ for some $S\subseteq[k_1]\times[k_2]$. Then by Definition~\ref{def:IS-for-k_1=2} (Steps~\ref{def: step loops},~\ref{def: step 2-minors}, and~\ref{def: step t minors}), we have that all $t$-minors of $A$ corresponding to the row slices of $[k_1]\times[k_2]$ and all $2$-minors of $A$ corresponding to the column slices of $[k_1]\times[k_2]$ vanish. Therefore, $A\in V(\mathcal{I}_C)$, implying that $V(I_S)\subseteq V(\mathcal{I}_C)$ for every $S\subseteq[k_1]\times[k_2]$.

For the other containment, let $A \in V(\mathcal{I}_C)$. Define $S := \{i : A_i = 0\}$ as the set of indices corresponding to the zero columns of $A$.

We will show that there exists some subset $S'\subseteq S$ such that $A\in V(I_{S'})$. We construct $S'$ iteratively by removing entries from the current zero set $S$ whenever a $(t-1)$-minor in $I_S$ fails to vanish on $A$. We then prove that $A \in V(I_{S'})$ by verifying that all minors of sizes 1, 2, $t-1$, and $t$ in $I_{S'}$ vanish on $A$.

The precise iterative construction is as follows. 

Let $S_0 := S$. For each $q \geq 1$, we analyze the ideal $I_{S_{q-1}}$ and define the next set $S_q$ as follows: 
\begin{itemize}
    \item[{\rm (i)}] If every $(t-1)$-minor in the generating set of ideal $I_{S_{q-1}}$ vanishes on $A$, we are done since $A\in V(I_{S_{q-1}})$, as shown below.
   \item[{\rm (ii)}]
    Otherwise, there exists a $(t-1)$-minor $[D \mid E]$, obtained either in Step~\ref{def: step t-1 minors} or when taking the closure of $H(S_{q-1})$, such that $[D \mid E]$ does not vanish on the matrix $A$. 
    Moreover, for some $i,j \in [k]$, either 
    $E$ is a subset of $\{i,j\} \times (NZ(i,S_{q-1}) \cap NZ(j,S_{q-1}))$, 
    or $E$ is identified via $S_{q-1}$ with a subset of this set. 
    Let $\mathcal{P}^{(q)}$ denote the set of all such pairs $\{i,j\}$. 
    For each pair $\{i,j\} \in \mathcal{P}^{(q)}$, define 
    \[
        L_{ij}^{(q)} \coloneqq \{i,j\} \times \bigl(Z(i,S_{q-1}) \,\Delta\, Z(j,S_{q-1})\bigr),
    \]
    where $Z(i,S_{q-1}) \,\Delta\, Z(j,S_{q-1})$ denotes the symmetric difference of the index sets $Z(i,S_{q-1})$ and $Z(j,S_{q-1})$. 
    Finally, set
    \[
        S_q := S_{q-1} \setminus \bigcup_{\{i,j\} \in \mathcal{P}^{(q)}} L_{ij}^{(q)}.
    \]
\end{itemize}
The above process terminates after finitely many steps, as elements are removed from the finite set $S$ at each step. Let $n$ denote the final step, and define $S' := S_n$.
We will now prove that $A \in V(I_{S'})$, i.e., all minors in the generating set of $I_{S'}$ vanish on $A$. 

\paragraph{Claim 1: all $(t-1)$-minors in the generating set of $I_{S'}$ vanish on $A$.}
In the process of constructing $S'$, and specifically in step (ii) above, for each $q\geq 1$, we removed every $(t-1)$-minor in the generating set of $I_{S_{q-1}}$ that does not vanish on $A$ by ensuring that the conditions $|Z(i,S_{q}) \setminus Z(j,S_{q})| \geq 1$ and $|Z(j,S_{q}) \setminus Z(i,S_{q})| \geq 1$ do not hold anymore for two indices $i$ and~$j$
from which that $(t-1)$-minor arises, as in Step 3 of Definition \ref{def:IS-for-k_1=2}. As a result, every $(t-1)$-minor in the generating set of $I_{S'}$ vanishes on $A$.

\paragraph{Claim 2: all $2$-minors in the generating set of $I_{S'}$ vanish on $A$.}

Since $A\in V(\mathcal{I}_C)$, all 2-minors of $A$ corresponding to the columns whose indices come from any column slice $C_i$ of $[k_1]\times[k_2]$ vanish.

\paragraph{Claim 3: all $t$-minors in the generating set of $I_{S'}$ vanish on $A$.}  

Note that all the $t$-minors in the generating set of $I_{S_0} = I_S$ vanish on $A$. This is due to the choice of $S$ and the fact that $A\in V(\mathcal{I}_C)$. Now assume all the $t$-minors in the generating set of  $I_{S_{q-1}}$ vanish on $A$ for some $q\geq 1$. We will show this is also true about all the $t$-minors in $I_{S_q}$. We prove this by contradiction. The key idea is that if some $t$-minor in the generating set of $I_{S_{q}}$ fails to vanish on $A$, then its support must involve positions identified with points removed from $S_{q-1}$ during the iterative construction of $S_{q}$. But each such removal is caused by the failure of a $(t-1)$-minor to vanish, which enforces new linear dependencies among the corresponding columns of $A$. We will show that these dependencies and column identifications force all the necessary $t$-minors of $A$ to vanish, leading to a contradiction.

Assume for contradiction that there exists a $t$-subset $U = \{a_1, \ldots, a_t\}$ in $\overline{H(S_{q})}$ such that some $t$-minor $g$ on the columns indexed by $U$ does not vanish on $A$. Assume that for each $u\in [t]$, $a_{u}$ belongs to the column slice $C_{i_u}$ of $[k_1]\times[k_2]$.
Since $U \in \overline{H(S_{q})}$, by construction, there exists some row index $r \in [k_1]$ such that each $a_u$ is identified with the entry $(r,i_u)$, where $(r,i_u) \notin S_{q}$. This follows from the construction rule in Step~\ref{def: step t minors} and the closure process: in order for $U$ to be added to $\overline{H(S_{q})}$, all its entries $a_u$ as well as  $(r,i_u)$ must lie outside $S_{q}$.

If $a_u\in S$ for some $u\in [t]$, then by the definition of $S$, we have $A_{a_u} = 0$, and therefore, the minor $g$ must vanish on $A$, a contradiction. Therefore, we can assume that $a_u\not\in S$ for all $u\in [t]$.

Now if for all $u\in [t]$, we have $(r,i_u) \not\in S_{q-1}$, then by Step~\ref{def: step t minors} of Definition~\ref{def:IS-for-k_1=2}, we have $\{(r,i_u) : u\in [t]\}\in H(S_{q-1})$. But for all $u\in [t]$, $a_u$ and $(r,i_u)$ get identified via $S_{q-1}$. So, this implies that $U \in \overline{H(S_{q-1})}$, a contradiction.
Hence, there must exist some $u \in [t]$ such that $(r,i_u) \in S_{q-1} \backslash S_{q}$. Define the index set  
$$
B \coloneqq \{u \in [t] : (r,i_u) \in S_{q-1} \backslash S_{q}\}.
$$  
For each $u \in B$, since $(r,i_u)$ was removed from $S_{q-1}$ during the construction of $S_{q}$, it must be that $(r,i_u) \in L_{r,r_u}^{(q)}$ for some $r_u \in [k]$ with $\{r, r_u\} \in \mathcal{P}^{(q)}$. But then, by the definition of $\mathcal{P}^{(q)}$, there exists a non-vanishing $(t-1)$-minor $[D_u\mid E_u]$ on $A$ such that $E_u$ is either a subset of the subgrid $\mathcal{G}_{\{r,r_u\}, NZ(r,S_{q-1}) \cap NZ(r_u, S_{q-1})}$ or is identified via $S_{q-1}$ with a subset of this subgrid. Therefore, the columns of $A_{E_u}$ are linearly independent. But now consider any point $b_u$ of the grid outside of $S_{q-1}$ and identified via $S_{q-1}$ with some point on the rows $r$ or $r_u$ of the grid. We have $E_u\cup \{b_u\} \in \overline{H(S_{q-1})}$. So, all the $t$-minors of $A_{E_u\cup \{b_u\}}$ vanish, meaning that $A_{E_u\cup \{b_u\}}$ has rank at most $t-1$, and therefore, the column space of $A_{E_u\cup \{b_u\}}$ is equal to the column space of $A_{E_u}$. Note that this implies the column spaces of $A_{E_u}$ are the same for all $u\in B$.
Now observe the following identifications:  
\begin{itemize}
    \item For $u \in B$, since $a_u \notin S_{q-1}$ and $(r_u, i_u) \notin S_{q-1}$, $a_u$ and $(r_u, i_u)$ are identified via $S_{q-1}$.  
\item For $u \notin B$, both $a_u$ and $(r, i_u)$ lie outside $S_{q-1}$, so $a_u$ and $(r, i_u)$ are identified via $S_{q-1}$.
\end{itemize}
These cases imply that $A_U$ is a subset of the column space of $A_{E_u}$ for any $u\in B$, and therefore, has rank at most $t-1$.
Hence, the minor $g$ must vanish on $A$, contradicting our initial assumption. This completes the proof of Claim 3.

\medskip
The three claims above show that every minor in the generating set of $I_{S'}$ vanishes on $A$. Therefore, $A \in V(I_{S'})$, as desired.
\end{proof}

The following example shows that the decomposition in \eqref{eq: decomposition} is not necessarily minimal.

\begin{example}\label{ex:bigger_k}
Not all subsets $S \subseteq [k_1]\times [k_2]$ are necessary in the decomposition \eqref{eq: decomposition}. For instance, let $k_1=2$, $k_2 = 5$, $t = 4$, and assume $d \geq t$. We claim that $I_{\{(1,1),\, (1,2)\}} \subseteq I_{\{(1,1),\, (1,2),\, (2,2)\}}$.  
\begin{itemize}
    \item For every $i \in [d]$, the variables $x_{i(1,1)}$ and $x_{i(1,2)}$ belong to $I_{\{\{(1,1),\, (1,2),\, (2,2)\}\}}$. 
    \item Moreover, the ideals $I_{\{(1,1),\, (1,2)\}}$ and $I_{\{(1,1),\, (1,2),\, (2,2)\}}$ share the same 2-minors in their generating sets, and neither includes any $(t-1)$-minors.  
\item 
The only $t$-minors appearing as minimal generators in $I_{\{(1,1),\, (1,2)\}}$ but not in $I_{\{(1,1),\, (1,2),\, (2,2)\}}$ are those corresponding to subsets containing $(2,2)$. However, each monomial in such a $t$-minor necessarily contains a variable $x_{i(2,2)}$ for some $i \in [d]$, and since $x_{i(2,2)} \in I_{\{(1,1),\, (1,2),\, (2,2)\}}$, the entire minor must also lie in $I_{\{(1,1),\, (1,2),\, (2,2)\}}$.
\end{itemize}
Despite this containment, neither $I_{\{(1,1),\, (1,2)\}}$ nor $I_{\{(1,1),\, (1,2),\, (2,2)\}}$ is a minimal component in 
\eqref{eq: decomposition}, as $I_\emptyset \subseteq I_{\{(1,1),\, (1,2)\}}$. 
On the other hand, in the inclusion $I_{\{(1,1),\, (2,2),\, (2,3)\}} \subseteq I_{\{(1,1),\, (2,2),\, (2,3),\, (2,4)\}}$, the ideal $I_{\{(1,1),\, (2,2),\, (2,3)\}}$ is a minimal prime component. A full characterization of the minimal components is discussed in Section~\ref{sec:minimality}.
\end{example}
   \begin{figure}[H]
       \centering
    \begin{tikzpicture}[scale=0.8]
    \tikzset{
        dot/.style={circle, fill, inner sep=0pt, minimum size=6pt},
        circled/.style={circle, draw, inner sep=4pt}
    }
        \foreach \x in {0,...,4}{
        \node[dot] at (\x,0) {};
        }
    \foreach \x in {0,...,4}{
        \node[dot] at (\x,1) {};
        \pgfmathtruncatemacro{\label}{int(\x + 1)};
        \node[label, font=\small] at (\x,1.6) {\label};
        }

    \node at (-0.8,1) {1};
    \node at (-0.8,0) {2};        
    \node[circled] at (0,1) {};
    \node[circled] at (1,1) {};
\end{tikzpicture}
\hspace{3cm}
    \begin{tikzpicture}[scale=0.8]
    \tikzset{
        dot/.style={circle, fill, inner sep=0pt, minimum size=6pt},
        circled/.style={circle, draw, inner sep=4pt}
    }
        \foreach \x in {0,...,4}{
        \node[dot] at (\x,0) {};
        }
    \foreach \x in {0,...,4}{
        \node[dot] at (\x,1) {};
        \pgfmathtruncatemacro{\label}{int(\x + 1)};
        \node[label, font=\small] at (\x,1.6) {\label};
        }
    \node at (-0.8,1) {1};
    \node at (-0.8,0) {2};
    \node[circled] at (0,1) {};
    \node[circled] at (1,0) {};
    \node[circled] at (1,1) {};

\end{tikzpicture}
\caption{Labeled grid $[2] \times [5]$ with $S = \{(1,1),\, (1,2)\}$ (left) and $S = \{(1,1),\, (1,2),\, (2,2)\}$ (right).}
\label{fig:nonminimal-example}
\end{figure}

The following example shows that for $k_1 \geq 3$, the minimal components appearing in the decomposition~\eqref{eq: decomposition} may no longer be prime.\begin{example} \label{example: I_S is not prime if k>=2}
Let $k_1 = 3$, $k_2 = 5$, $t = 4$, and $d=4$. Fix $S = \{6, 11, 13\}$.

\begin{figure}[H]
\centering
    \begin{tikzpicture}[scale=0.8]
    \tikzset{
        dot/.style={circle, fill, inner sep=0pt, minimum size=6pt},
        circled/.style={circle, draw, inner sep=4pt}
    }
        \foreach \x in {0,...,4}{
        \node[dot] at (\x,0) {};
        }
    \foreach \x in {0,...,4}{
        \node[dot] at (\x,1) {};
        }

    \foreach \x in {0,...,4}{
        \node[dot] at (\x,2) {};
        \pgfmathtruncatemacro{\label}{int(\x + 1)};
        \node[label, font=\small] at (\x,2.6) {\label};
        }
    \node at (-0.8,2) {1};
    \node at (-0.8,1) {2};
    \node at (-0.8,0) {3};
    \node[circled] at (1,0) {};
    \node[circled] at (3,1) {};
    \node[circled] at (4,2) {};

\end{tikzpicture}
\caption{The labeled grid $[3] \times [5]$ for $S=\{(1,5), \, (2,4), \, (3,2)\}$.}
\label{fig:k=3-hypergraph-example}
\end{figure}
The ideal $I_S$ is minimal among $\{I_S : S \subseteq [k_1]\times [k_2]\}$, as it can be verified that for any $S' \subsetneq S$, we have $I_{S'} \not\subseteq I_S$.\footnote{\url{https://github.com/yuliaalexandr/decomposing-conditional-independence-ideals-with-hidden-variables}} However, $I_S$ is not prime and decomposes into two prime components, $I_1$ and~$I_2$:
\begin{itemize}
    \item The ideal $I_1 = I(\overline{H_1})$ is the ideal of the closure of the hypergraph $$H_1 := H(S) \cup \{2\text{-subsets of } \{1,2,3\} \times \{1,3\} \}.$$

\item  The ideal $I_2=I(\overline{H}_2)$, where the hypergraph $H_2$ is obtained from $H(S)$ by adding all ($t-1$)-subsets of the unions of the ($t-1$)-subsets in $\overline{H(S)}$ that intersect in at least $t-2$ elements,~i.e.,
\begin{align*}
H_2 := H(S) \cup \left\{ \{i_{k_1}, \ldots, i_{k_{t-1}}\} : 
\begin{array}{l}
\text{for any } \{i_1, i_2, \ldots, i_{t-1}\}, \{i_2, \ldots, i_{t-1}, i_t\} \in \overline{H(S)}, \\
\text{with } i_1, \ldots, i_t \text{ all distinct,} \\
\text{and distinct indices } k_1, \ldots, k_{t-1}\in[t]
\end{array}
\right\}.
\end{align*}
where each $i_j$ is a pair of indices in $[k_1]\times [k_2]$.
\end{itemize}

\end{example}

\section{Minimal ideals $I_S$ in the decomposition of $\mathcal{I}_C$} \label{sec:minimality}

Throughout this section, we fix $k_1 = 2$ and $2\leq t \leq \min\{k_2,d\}$. Our main goal in this section is to characterize the subsets $S \subseteq [2]\times [k_2]$ for which the ideal $I_S$ is a minimal component in the decomposition~\eqref{eq: decomposition}. We also provide numerical counts for the number of such subsets $S$.

\subsection{Combinatorial characterization of minimal ideals}



Our first task is describe the sets $S$ such that $I_S$ is minimal in the decomposition~\eqref{eq: decomposition}.
We first prove that $I_S$ is not minimal if $S$ contains any column $C_i$.
\begin{lemma}\label{lem:min-cond2}
 If $C_i \subseteq S$ for some $i \in [k_2]$, then $I_S$ is not minimal among $\{I_S : S \subseteq [2]\times[k_2]\}$.
\end{lemma}

\begin{proof}
   Suppose $S \subseteq [2]\times [k_2]$ is such that $C_i \subseteq S$ for some $i \in [k_2]$. If $Z(1,S) \subseteq Z(2,S)$ or $Z(2,S) \subseteq Z(1,S)$, then $I_\emptyset \subsetneq I_S$, and therefore, $I_S$ is not minimal. Hence, we can assume 
   \begin{align} \label{eq: t-1 condition}
       |Z(1,S) \setminus Z(2,S) | \geq 1 \text{ and } |Z(2,S) \setminus Z(1,S) | \geq 1.
   \end{align} 
   Define $S' = S \backslash \{(2,i)\}$, and note that \eqref{eq: t-1 condition} holds for $S'$ as well. Let $\mathcal{C}(S)$ denote the union of all column slices $C_j$ such that $C_j \cap S = \emptyset$.  Then, by Definition~\ref{def:k_1=2}, the ideals $I_S$ and $I_{S'}$ are generated by the sets $F_S$ and $F_{S'}$, respectively, where
\begin{align*}
        F_S &= F_1(X_S) \cup  \bigcup_{\substack{j \in [k_2],\\ C_j \cap S = \emptyset}} F_2(X_{C_j}) \cup F_{t-1}(X_{\mathcal C(S)}) \cup F_t(X_{R_1 \cup \mathcal C(S)}) \cup F_t(X_{R_2 \cup \mathcal C(S)}),\\
        F_{S'} &= F_1(X_{S'}) \cup  \bigcup_{\substack{j \in [k_2],\\ C_j \cap S' = \emptyset}} F_2(X_{C_j}) \cup F_{t-1}(X_{\mathcal C(S')}) \cup F_t(X_{R_1 \cup \mathcal C(S')}) \cup F_t(X_{R_2 \cup \mathcal C(S')}).
    \end{align*}   
    Note that $C_j \cap S = \emptyset$ if and only if $C_j \cap S' = \emptyset$, implying $\mathcal{C}(S) = \mathcal{C}(S')$. Consequently, we have $F_S \backslash F_{S'} = F_1(X_S) \backslash F_1(X_{S'}) = \{x_{j(2,i)}: j\in [d]\}$, which yields the desired strict inclusion $I_{S'} \subsetneq I_S$.
\end{proof}

Since sets $S$ containing column slices of $[2] \times [k_2]$ are never minimal, we focus on the sets $S$ which do not contain a column slice. 
Our next step will be to group the remaining sets $S$ into equivalence classes where two sets $S$ and $S'$ are considered equivalent if their associated ideals $I_S$ and $I_{S'}$ are isomorphic up to a relabeling of variables. 
This notion of equivalence coincides with the notion of combinatorial equivalence:
\begin{definition}
    Suppose $S\subseteq[2]\times[k_2]$ is a subset such that $C_i\not\subseteq S$ for any $i\in[k_2]$. Let
    \[u=\min\{|R_1\cap S|,|R_2\cap S|\},\qquad v=\max\{|R_1\cap S|,|R_2\cap S|\}.\]
    We call the pair $(u,v)$ the \emph{combinatorial type} of $S$.
\end{definition}

The goal of the rest of this subsection will be to characterize the minimal $I_S$ by combinatorial type. 
This characterization is given by the following definition of a \emph{minimal subset} $S\subseteq[2]\times[k_2]$. We will then
give a series of results that show that $S$ being a minimal subset is equivalent to $I_S$ being minimal among the set of ideals $\{I_S:S\subseteq[2]\times [k_2]\}$.

\begin{definition}\label{def:min}
    A set $S\subseteq[2]\times [k_2]$ with combinatorial type $(u,v)$ is called \emph{minimal} if $S=\emptyset$ or all of the following conditions hold:
    \begin{enumerate}
        \item $C_i\not\subseteq S$ for all $i\in[k_2]$;\label{def:min 1}
        \item $1\leq u\leq v\leq k_2-t+1$;\label{def:min 2} 
        \item If $t=2$, we further require that $u+v=k_2$.
    \end{enumerate}
\end{definition}


We now prove in a series of results that $S$ is minimal if and only if the ideal $I_S$ is minimal. 

\begin{lemma}\label{lem:t=2 min con 3}
    Let $t=2$. Suppose $S$ is nonempty and has combinatorial type $(u,v)$. If $u+v<k_2$, then $I_S$ is not minimal among $\{I_S:S\subseteq[2]\times[k_2]\}$.
\end{lemma}

\begin{proof}
    Since $u+v<k_2$, the set $\mathcal{C}(S)$ is nonempty. Let $S'=S\cup\mathcal{C}(S)$. Then $I_S=I_{S'}$, as both ideals are generated by
    \[F_1(X_S)\cup F_1(X_{\mathcal{C}(S)})\cup F_2(X_{R_1})\cup F_2(X_{R_2}).\]
    Now, $I_{S'}$ is not minimal by Lemma~\ref{lem:min-cond2}, and so $I_S$ cannot be minimal as well.
\end{proof}

\begin{lemma}\label{lem:min-cond1}
  Suppose that $S$ has combinatorial type $(0, v)$ with $1 \leq v \leq k_2$. Then the corresponding ideal $I_S$ is not minimal among the collection $\{I_S : S \subseteq [2]\times[k_2]\}$.
\end{lemma}

\begin{proof}
   Suppose that $S$ has combinatorial type $(0,v)$ for $1 \leq v \leq k_2$. We show that $I_\emptyset \subsetneq I_S$ by verifying that every generator of $I_\emptyset$ lies in $I_S$. Then since $I_S$ contains variables as generators, specifically those corresponding to the columns indexed by $S$, while $I_\emptyset$ does not include any degree-one generators, the inclusion is strict. The generators of $I_\emptyset$ consist of 2-minors and $t$-minors. The 2-minors, arising from 2-subsets within each column, are all contained in $I_S$, either directly or via the presence of degree-one generators. Similarly, all $t$-minor generators of $I_\emptyset$ are contained in $I_S$, either directly or as consequences of its degree-one or 2-minor generators. Thus $I_\emptyset \subsetneq I_S$. 
\end{proof}

\begin{lemma}\label{lem:non-min}

Suppose $S$ has combinatorial type $(u,v)$ such that either $u>k_2-t+1$ or $v>k_2-t+1$. Then $I_S$ is not minimal among the set of ideals $\{I_S:S\subseteq[2]\times[k_2]\}$.
 \end{lemma}

\begin{proof}
Suppose without loss of generality that $v>k_2-t+1$, with $u=|R_1\cap S|$ and $v=|R_2\cap S|$. Let $S'=S\setminus \{a\}$, for some $a\in C_i\cap R_2\cap S$, so $S'$ has combinatorial type either $(u,v-1)$ or $(v-1,u)$. We will show that $I_{S'}\subsetneq I_S$.

First, note that $u+v>k_2-t+2$, and so the number of columns $C_j$ for which $C_j\cap S=\emptyset$ is strictly less than $k_2-(k_2-t+2)=t-2$. Similarly, the number of columns $C_j$ for which $C_j\cap S'=\emptyset$ is strictly less than $t-1$. Thus, the generating sets $F_S$ and $F_{S'}$ of $I_S$ and $I_{S'}$ respectively both do not contain non-trivial $(t-1)$-minor generators. Furthermore, $|R_2\setminus S|<t$ and $|R_2\setminus S'|<t$, and so there are no non-trivial $t$-minor generators coming from $(R_2\setminus S)\cup\mathcal{C}(S)$ and $(R_2\setminus S')\cup\mathcal{C}(S')$.

Then the ideals $I_S$ and $I_{S'}$ are generated by the sets $F_S$ and $F_{S'}$ respectively:
    \begin{align*}
        F_S &= F_1(X_S) \cup  \bigcup_{\substack{j \in [k_2]\\ C_j \cap S = \emptyset}} F_2(X_{C_j}) \cup F_t(X_{(R_1\setminus S) \cup \mathcal C(S)}),\\
        F_{S'} &= F_1(X_{S'}) \cup  \bigcup_{\substack{j \in [k_2]\\ C_j \cap S' = \emptyset}} F_2(X_{C_j}) \cup F_t(X_{(R_1\setminus S') \cup \mathcal C(S')}).
    \end{align*} 
Since $S' \neq S$, the ideal $I_S$ contains a variable generator that is not in $I_{S'}$, so $I_{S'} \neq I_S$. Now, observe that since $S' \subseteq S$, we have $F_1(X_{S'}) \subseteq F_1(X_S)$. For the 2-minor generators in $F_{S'}$, the set $F_2(X_{C_i})$ is the only one not in $F_S$. However, since $F_S$ includes the variable generators $F_1(X_{\{a\}})$, we have $F_2(X_{C_i}) \subseteq I_S$. Similarly, the only $t$-minor generators in $F_{S'}$ that are not in $F_S$ are those involving $F_1(X_{\{a\}})$. Thus, we have $F_t(X_{R_1 \cup \mathcal{C}(S')}) \subseteq I_S$. 
Since each generator of $I_{S'}$ is contained in $I_S$, we conclude that $I_{S'} \subsetneq I_S$. 
\end{proof}

\begin{proposition}\label{prop:min}
    If $S\subseteq[2]\times[k_2]$ is minimal, then $I_S$ is minimal among the set of ideals $\{I_S:S\subseteq[2]\times[k_2]\}$.
\end{proposition}

\begin{proof}
We first explain why $I_\emptyset$ is minimal. If $S\subseteq[2]\times[k_2]$ with $S$ nonempty, then $I_S$ contains a generator of degree one. Since $I_\emptyset$ is generated by polynomials of degrees $2$ and $t\geq 2$, this shows that $I_S\not\subseteq I_\emptyset$ for any $S$.

Now assume $S$ is nonempty and minimal as in Definition~\ref{def:min}, and it has combinatorial type~$(u, v)$. 
It suffices to prove the claim for a single representative of this combinatorial type. 
We choose 
$S = \{(1, 1), \ldots, (1,u), (2, k_2 - v + 1), (2, k_2-v+2), \ldots, (2, k_2)\}$.
Suppose there exists $S' \subseteq [2]\times [k_2]$ such that $I_{S'} \subsetneq I_S$. 
If $t=2$, this is not possible: by Lemma~\ref{lem:t=2 min con 3}, we may assume that the combinatorial type of $S'$ is $(u',v')$, where $u'+v'=k_2$ and so since $u+v=k_2$ as well, we see that $S=S'$ are actually the same subsets (the degree one elements of $I_S$ and $I_{S'}$ would determine $S$ and $S'$ respectively).

For the rest of the proof, suppose $t\geq 3$. Then we must have $S' \subsetneq S$, because the degree one elements in $I_S$ and $I_{S'}$ determine $S$ and $S'$ respectively. We will construct an element of $I_{S'}$ which is not contained in $I_S$. 
Suppose, without loss of generality, that $(2,i) \in S \setminus S'$ is in the second row slice $R_2$ and $i$th column slice $C_i$ of $[2] \times [k_2]$; the argument where an element of $S\setminus S'$ is in the first row slice is analogous.

Let
$T = \{(2,1),(2,2),\ldots,(2,t-1), (1,i)\}$ 
By Condition \ref{def:min 2}, we have that 
$t-1< i$.
We claim that the $t$-minor $f = [1 \cdots t| (2,1),(2,2),\ldots,(2,t-1), (1,i)]$ is 
in $I_{S'} \setminus I_S$.

We first show that $f \in I_{S'}$. 
Since $(2,i) \in S$, by Condition \ref{def:min 1} of the minimality definition, $(1,i) \notin S$ and hence $(1,i) \notin S'$. 
Thus $C_i \subseteq \mathcal C(S')$. 
So $f \in I_t(X_{R_2 \cup \mathcal C(S')}) \subseteq I_{S'}$. 

We now show that $f \notin I_{S}$.
Recall that $I_S$ is generated by
\begin{align*}
    F_S = F_1(X_S) \cup \bigcup_{\substack{j \in [k_2] \\ C_j \cap S = \emptyset}} F_2(X_{\{2j-1, 2j\}}) \cup F_{t-1}(X_{\mathcal C(S)}) \cup F_t(X_{(R_1 \setminus S) \cup \mathcal C(S)})
    \cup 
    F_t(X_{(R_2 \setminus S) \cup \mathcal C(S)}).
\end{align*}

Suppose for contradiction that $f = \sum_{g \in F_S}h_gg$ for some $h_g \in \CC[X]$. 
Then some $g \in F_S$ contains a monomial which divides the term  $x_{1(2,1)} x_{2(2,2)} \cdots x_{t-1,(2,t-1)} x_{t,(1,i)}$ in $f$.
Hence, it suffices to prove that no monomial in any $g \in F_S$ divides $x_{1(2,1)} x_{2(2,2)} \cdots x_{t-1,(2,t-1)} x_{t,(1,i)}$.
This is equivalent to proving that none of the elements of $\overline{H(S)}$ are contained in $T = \{(2,1),(2,2),\ldots,(2,t-1), (1,i)\}$.
Indeed,
\begin{itemize}
    \item we chose the elements of $T$ to be nonzero, so $T \cap S = \emptyset$;
    \item no column $C_i$ for which $C_i\cap S=\emptyset$ is contained in $T$, since $R_1\cap T=\{(1,i)\}$, and $(2,i)\notin T$;
    \item it is not possible for a $(t-1)$-subset of $\mathcal{C}(S)$ to be in $T$, since $u > 0$ implies that at most the $t-2$ elements $(2,2),\ldots,(2,t-1)$ of $T$ can be contained in $\mathcal{C}(S)$;
    \item it is not possible for a $t$-subset of $(R_1\setminus S)\cup\mathcal{C}(S)$ to equal $T$, since $(2,1)\notin R_1\cup\mathcal{C}(S)$;
    \item it is not possible for a $t$-subset of $(R_2\setminus S)\cup\mathcal{C}(S)$ to equal $T$, since $(1,i)\notin R_2\cup\mathcal{C}(S)$.\qedhere
\end{itemize}
\end{proof}
In the following theorem, we refine the decomposition from Theorem~\ref{thm: decomposition theorem} by eliminating redundant components and presenting a minimal decomposition in terms of the ideals $I_S$.

\begin{theorem}\label{thm:set_min_primes}
In the case of $k_1 = 2$ and $2\leq t \leq \min\{k_2,d\}$, the minimal decomposition of $\sqrt{\mathcal{I}_C}$ is given as 
$\sqrt{\mathcal{I}_C} = \bigcap \sqrt{I_S},$
where the intersection is taken over all minimal $S$.

\end{theorem}
\begin{proof}
   By Theorem \ref{thm: decomposition theorem}, we have the decomposition $\sqrt{\mathcal{I}_C} = \bigcap_{S \subseteq [k_1] \times [k_2]} \sqrt{I_S}$.
   It suffices to take the intersection over all minimal ideals. 
   By Lemmas~\ref{lem:min-cond2}, \ref{lem:t=2 min con 3}, \ref{lem:min-cond1}, and \ref{lem:non-min}, if $S$ is not minimal, then $I_S$ is not minimal and can be removed from the intersection. 
   By Proposition \ref{prop:min}, if $S$ is minimal, then $I_S$ is minimal and should be included in teh itnersection. 
\end{proof}

\subsection{Counting minimal ideals}

We now provide closed formulas for the number of minimal ideals $I_S$ that appear in the decomposition given in Theorem~\ref{thm:set_min_primes}. First, we count the distinct combinatorial types $(u,v)$ that arise in the minimal decomposition, where each type corresponds to an isomorphism class of ideals.

\begin{proposition}\label{cor:num isom class}
If $t=2$, the number of combinatorial types in the minimal decomposition of $\mathcal{I}_C$ is $\lfloor \frac{k_2}{2} \rfloor + 1$. If $t\geq 3$, the number of combinatorial types in the minimal decomposition of~$\mathcal{I}_C$~is 
$$\begin{cases}
    \tfrac{1}{2}(k_2-t+1)(k_2-t+2)+1  &\text{if $t-1 \geq \frac{k_2}{2}$,}\\
    \tfrac{1}{4}(k_2^2-2t^2+6t)  &\text{if $t-1<\frac{k_2}{2}$ and $k_2$ is even,}\\
     \tfrac{1}{4}(k_2^2 -2 t^2 +6t -1), &\text{if $t-1<\frac{k_2}{2}$ and $k_2$ is odd}.
\end{cases}$$

\end{proposition}

\begin{proof}
First consider the $t=2$ case. By Theorem~\ref{thm:set_min_primes} and using Definition \ref{def:min} we count the number of pairs $(u,v)$ for which $u+v=k_2$ and $1\leq u\leq v\leq k_2-1$. These conditions imply that $u$ can take on any value in $\{1,2,\ldots,\lfloor k_2/2\rfloor\}$, and that the value of $v$ is then determined. Note that $S$ may also be the empty set, and so we add $1$ to the final count.

Now suppose $t\geq 3$.
By Theorem~\ref{thm:set_min_primes} and using Definition \ref{def:min}, we count the number of pairs \((u,v)\) in the set
\begin{eqnarray}\label{eq:combin_type}
    \mathcal{S} = \{(u, v) : 1 \leq u \leq v \leq k_2 - t + 1 \text{ and } u + v \leq k_2\} \cup \{(0,0) \}.
\end{eqnarray}

Consider the case where \( t - 1 \geq \frac{k_2}{2} \), which is equivalent to \( \frac{k_2}{2} \geq k_2 - t + 1 \). This implies that for all integers $u\in [1, k_2-t+1]$ and $v \in [u, k_2-t+1]$, we have $u+v\leq 2(k_2-t+1) \leq k_2$, and hence, $(u,v) \in \mathcal{S}$. The number of combinatorial types in this case is given by the sum of the first \( k_2 - t + 1 \) natural numbers, as shown in Figure~\ref{counts:case1}, plus 1 to account for the empty set.

\begin{figure}[H]
\centering
$ 
 \begin{array}{cccccc}
(1,1)&(1,2) & (1,3)&\ldots &(1,k_2-t)&(1,k_2-t+1)\\
&(2,2)&(2,3)&\ldots&(2,k_2-t)&(2,k_2-t+1) \\
  & & (3,3)&\ldots &(3,k_2-t)&(3,k_2-t+1) \\
& && \ddots& \vdots & \vdots\\ 
& & & &(k_2-t,k_2-t)&(k_2-t,k_2-t+1) \\
& & & &&(k_2-t+1,k_2-t+1) \\
\end{array}
$
    \caption{Counting isomorphism classes for $t-1\geq k_2/2$}
    \label{counts:case1}
\end{figure}

Now, consider the case where $t - 1 < \frac{k_2}{2}$. In this case, $u$ can take on any integer value in $[1,\lfloor \frac{k_2}{2} \rfloor]$. For the bounds on $v$, when $u \leq t - 1$, $v$ attains any integer value in  $[u,k_2 - t + 1$], and when $u > t - 1$, the upper bound changes to $k_2 - u$. 
To count the isomorphism classes, we partition the classes into three sets, as illustrated in Figure~\ref{counts:case2-1}.

\begin{figure}[H]
\centering
\resizebox{\textwidth}{!}{$\displaystyle
    \begin{array}{ccccc|cccccc}
(1,1)&(1,2) & \ldots& (1,\lfloor \ell_2/2\rfloor-1)& (1,\lfloor \ell_2/2\rfloor) & (1,\lfloor \ell_2/2\rfloor+1) & (1,\lfloor \ell_2/2\rfloor+2) & \cdots& (1,k_2-t)  & (1,k_2-t+1)\\
&(2,2) & \cdots &(2,\lfloor \ell_2/2\rfloor-1)& (2,\lfloor \ell_2/2\rfloor) & (2,\lfloor \ell_2/2\rfloor+1) & (2,\lfloor k_2/2\rfloor+2) & \cdots& (2, k_2-t)   & (2, k_2-t+1)\\
&& \vdots& \vdots & \vdots &\vdots&\vdots&&\vdots &\vdots\\
&& &(t-1,\lfloor k_2/2\rfloor-1)&(t-1,\lfloor k_2/2\rfloor)& (t-1,\lfloor k_2/2\rfloor+1) & (t-1,\lfloor k_2/2\rfloor+2) & \cdots& (t-1, k_2-t)  & (t-1, k_2-t+1)\\
\cline{6-10}
&&  &(t,\lfloor k_2/2\rfloor-1)&(t,\lfloor k_2/2\rfloor)& (t,\lfloor k_2/2\rfloor+1) & (t,\lfloor k_2/2\rfloor+2) & \cdots& (t, k_2-t)   & \\
&& \ddots &\vdots& \vdots & \vdots &\vdots &~\reflectbox{$\ddots$}  & \\
&&&(\lfloor k_2/2\rfloor-1,\lfloor k_2/2\rfloor-1)& (\lfloor k_2/2\rfloor-1,\lfloor k_2/2\rfloor)&(\lfloor k_2/2\rfloor-1,\lfloor k_2/2\rfloor+1)&\boxed{(\lfloor k_2/2\rfloor-1,\lfloor k_2/2\rfloor+2)} & &&\\
&&&&(\lfloor k_2/2\rfloor,\lfloor k_2/2\rfloor)& \boxed{(\lfloor k_2/2\rfloor,\lfloor k_2/2\rfloor+1)} &  & &  &  \\
\end{array}$
    }
\caption{Counting isomorphism classes for $t-1< k_2/2$. The boxed items are counted if $ k_2$ is~odd.}
\label{counts:case2-1}
\end{figure}

The left block corresponds to the region where $u \leq v \leq \lfloor \frac{ k_2}{2} \rfloor$, and contributes the sum of the first $\lfloor \frac{ k_2}{2} \rfloor$ natural numbers.
Next, in the  upper right block, for $1 \leq u \leq t - 1$, we consider those pairs with $\lfloor \frac{ k_2}{2} \rfloor + 1 \leq v \leq  k_2 - t + 1$. This adds $(t - 1)( k_2 - t + 1 - \lfloor \frac{ k_2}{2} \rfloor)$ to the total count.
Finally, the  lower right block  accounts for the remaining combinatorial types, and it splits into two cases depending on the parity of $ k_2$. If $ k_2$ is even, we consider those with $t \leq u \leq \lfloor \frac{ k_2}{2} \rfloor - 1$ and $\lfloor \frac{ k_2}{2} \rfloor + 1 \leq v \leq  k_2 - u$, which contributes the sum of the first $\frac{ k_2}{2} - t$ natural numbers. If $ k_2$ is odd, we instead consider $t \leq u \leq \lfloor \frac{ k_2}{2} \rfloor$ and the same bound on $v$, giving the sum of the first $\lceil \frac{ k_2}{2} \rceil - t$ natural numbers.
We then add 1 more isomorphism class to account for the empty set.

In total, if $ k_2$ is even, the three blocks plus the empty set contribute
\[
\frac{\tfrac{ k_2}{2}(\tfrac{ k_2}{2}+1)}{2} + (t-1)\left(\tfrac{ k_2}{2} - t + 1\right) + \frac{(\tfrac{ k_2}{2} - t)(\tfrac{ k_2}{2} - t + 1)}{2} + 1 = \frac{1}{4}( k_2^2 - 2t^2 + 6t).
\]
If $ k_2$ is odd, the three blocks plus 1 sum to \[\frac{\tfrac{ k_2-1}{2}(\tfrac{ k_2-1}{2}+1)}{2} + (t-1)(\tfrac{ k_2+1}{2}-t+1) + \frac{(\tfrac{ k_2+1}{2}-t)(\tfrac{ k_2+1}{2}-t+1)}{2} +1 = \frac{1}{4}( k_2^2 -2t^2 +6t -1).\qedhere\]
\end{proof}

We now count the number of subsets $S\subseteq[2]\times [k_2]$ with a given combinatorial type. This, together with Proposition~\ref{cor:num isom class}, gives the exact number of ideals $I_S$ needed to decompose $\mathcal{I}_C$.

\begin{corollary}
    The number of minimal components with combinatorial type $(u,v)$ is given by
    \begin{align*}
        \binom{ k_2}{u}\binom{ k_2-u}{v}\quad\text{if $u=v$} && \text{and} &&
        2\binom{ k_2}{u}\binom{ k_2-u}{v}\quad\text{if $u\neq v$.}
    \end{align*}
\end{corollary}

\begin{proof}
    For each combinatorial type $(u,v)$, the elements selected for $u$ and $v$ must come from distinct columns $C_i$. First, choose $u$ columns from the $ k_2$ available, and then choose $v$ columns from the remaining $ k_2 - u$. This yields $\binom{ k_2}{u} \binom{ k_2 - u}{v}$ configurations. When $u\neq v$, we multiply by $2$ to account for the choice of which row contains $u$ elements in $S$.
\end{proof}

\section{Gr\"obner bases, radicality, and primeness of ideals $I_S$}
\label{sec:primeness}
In this section, we focus on minimal ideals $I_S$, i.e., the ones appearing in Theorem~\ref{thm:set_min_primes}.
We prove that the natural generating set $F_S$ for $I_S$ is a squarefree Gr\"{o}bner basis, as detailed in Corollary~\ref{prop:ISgb}. 
We then conclude that the ideals $I_S$  are radical. 
In the case  $d = t$, we prove the ideals $I_S$ are also prime by giving a polynomial parametrization from an irreducible variety; see Theorem~\ref{thm:ISprime}.
This shows that when $d = t$, the decomposition in Theorem~\ref{thm:set_min_primes} is a prime one. 
Our results are stronger for $S = \emptyset$: when $2\leq k_1$ and $2 \leq t \leq \min\{ k_2,d\}$, then $F_\emptyset$ is a Gr\"obner basis and $I_\emptyset$ is prime; this is the content of Theorems~\ref{prop:Iemptygb} and \ref{thm:param-empty}.

\medskip
Before proceeding with our main results, we recall some notation. Let $X$ denote the $d \times k_1k_2$ matrix of indeterminates. 
The columns of $X$ are ordered reverse lexicographically:
\[(1,1),(2,1),\ldots,(k_1,1),(1,2), (2,2),\ldots, (k_1, 2), (1,3), (2,3), \ldots,(k_1,k_2).\]
For the variables of our ring $R$, we use the lexicographic term order $\prec$ with respect to the total variable order
\[ x_{1(1,1)} > x_{2(1,1)} > \cdots > x_{d(1,1)} > x_{1(2,1)} > \cdots > x_{d,(k_1, k_2)}. \]
Additionally, $X_U$ denotes the submatrix of $X$ whose columns are indexed by the set $U$. For consecutive columns $(a,b),\ldots,(c,d)$ of our matrix $X$, we denote the corresponding submatrix by $X[(a,b),(c,d)] = X_{\{(a,b),\ldots,(c,d)\}}$.

\medskip
We now proceed with a key reduction. Recall the definition of the ideal $I_S$ and its generating set $F_S$ for $k_1=2$ from Definition~\ref{def:k_1=2}. For a minimal nonempty $S$, let $J_S$ denote the ideal generated by the following sets of minors in $F_S$:
\begin{align}\label{eq:FJS}
F(J_S) \coloneqq \bigcup_{\substack{i \in [ k_2] \\ C_i \cap S = \emptyset}} F_2(X_{C_i}) \cup F_{t-1}(X_{\mathcal C(S)}) \cup F_t(X_{(R_1 \backslash S) \cup \mathcal C(S)}) \cup F_t(X_{(R_2 \backslash S) \cup \mathcal C(S)}).
\end{align}
For any $k_1 \geq 2$, we have $C_i = [k_1] \times \{i\}$ and we define
\begin{align}\label{eq:Iemptygens}
    J_{\emptyset} \coloneqq I_\emptyset = I_t(X) + \sum_{i = 1}^{ k_2} I_2(X_{C_i})\quad\text{and}\quad 
        F_\emptyset \coloneqq F_t(X) \cup \bigcup_{i=1}^{ k_2} F_2(X_{C_i}).
\end{align}
\begin{remark}\label{rem:J_S}
    We observe that $I_S = I_1(X_S) + J_S$, where $I_1(X_S)$ is the ideal generated by the variables corresponding to the columns indexed by $S$. We claim that it suffices to prove that $J_S$ is prime in order to show that $I_S$ is prime. Specifically, since the ideal $I_1(X_S)$ is prime and the generators of $J_S$ and $I_1(X_S)$ involve disjoint sets of variables, if $J_S$ is prime, then their sum, $I_S = I_1(X_S) + J_S$, must also be prime.
  Moreover, a Gr\"{o}bner basis of $J_S$ can be extended to that of $I_S$ by adding the variables in $F_1(X_S)$.
\end{remark}

Since the generators of $J_S$ do not involve variables from the columns indexed by $S$, we define the matrix $\hat{X}$ as the matrix $X$ with the columns indexed by $S$ removed. We index the columns of $\hat{X}$ according to the original indexing of $X$. 

\medskip
We now outline our strategy. 
For each combinatorial type of minimal $S$, as in Definition~\ref{def:min}, there exists a representative ideal $J_{S'}$ whose natural generating set forms a Gr\"obner basis with respect to the chosen term order. We use this particular representative to prove that $J_S$ is radical for any set $S$ of the same combinatorial type. We illustrate this approach with the following example.

\begin{example}\label{example_Sec4}
    Let $k_1=2, t=4, k_2 = 6,d=4$, and $S = \{1\} \times \{2,4\} \cup \{2\} \times \{3,5\}$.
    This set has combinatorial type $(2, 2)$. This setup corresponds to the following grid, where the elements of $S$ are circled.
    \begin{center}
    \begin{tikzpicture}[scale=0.8]
    \tikzset{
        dot/.style={circle, fill, inner sep=0pt, minimum size=6pt},
        circled/.style={circle, draw, inner sep=4pt}
    }
        \foreach \x in {0,...,5}{
        \node[dot] at (\x,0) {};
        \node[label, font=\small] at (\x,-0.6) {};
        }
    \foreach \x in {0,...,5}{
        \node[dot] at (\x,1) {};
        \pgfmathtruncatemacro{\label}{int(\x + 1)};
        \node[label, font=\small] at (\x,1.6) {\label};
        }
           \node at (-0.8,1) {1};
    \node at (-0.8,0) {2};   
    \node[circled] at (1,1) {};
    \node[circled] at (2,0) {};
    \node[circled] at (3, 1) {};
    \node[circled] at (4, 0) {};
\end{tikzpicture}
\end{center}
    Then the ideal $J_S$ is generated by $F_2(X_{C_1}) \cup F_2(X_{C_6}) \cup F_{3}(X_{\{1,2\} \times \{1,6\}}) \cup F_4(X_{\{1\} \times \{3,5\} \cup \{1,2\} \times \{1,6\}}) \cup F_4(X_{\{2\} \times \{2,4\} \cup \{1,2\} \times \{1,6\}})$. 
    This set is not a Gr\"{o}bner basis with respect to the order $\prec$. 
    However, we may instead consider $S' = \{1\} \times \{1,2\} \cup \{2\} \times \{5,6\}$, which has the same combinatorial type.
    \begin{center}
    \begin{tikzpicture}[scale=0.8]
    \tikzset{
        dot/.style={circle, fill, inner sep=0pt, minimum size=6pt},
        circled/.style={circle, draw, inner sep=4pt}
    }
    
    \foreach \x in {0,...,5}{
        \node[dot] at (\x,0) {};
        \node[label, font=\small] at (\x,-0.6) {};
        }
    \foreach \x in {0,...,5}{
        \node[dot] at (\x,1) {};
        \pgfmathtruncatemacro{\label}{int(\x + 1)};
        \node[label, font=\small] at (\x,1.6) {\label};
        }
           \node at (-0.8,1) {1};
    \node at (-0.8,0) {2};   
    \node[circled] at (0,1) {};
    \node[circled] at (1,1) {};
    \node[circled] at (4, 0) {};
    \node[circled] at (5, 0) {};
\end{tikzpicture}
\end{center}
    The natural generating set of $J_{S'}$ is 
    \[ F_2(\hat{X}_{C_3}) \cup F_2(\hat{X}_{C_4}) \cup F_{3}(\hat{X}_{\{1,2\} \times \{3,4\}}) \cup F_4(\hat{X}_{\{2\} \times \{1,2\} \cup \{1,2\} \times \{3,4\}}) \cup F_4(\hat{X}_{\{1\} \times \{5,6\} \cup \{1,2\} \times \{3,4\}}),\] 
    where $\hat{X}$ is the matrix 
     \begin{equation*} 
     \hat{X} = \begin{pNiceMatrix}[r,left-margin=0.6em, right-margin=0.3em]
    x_{1(2,1)}  & x_{1(2,2)}  & x_{1(1,3)}  & x_{1(2,3)} & x_{1(1,4)} & x_{1(2,4)} & x_{1(1,5)} & x_{1(1,6)} \\
    x_{2(2,1)}  & x_{2(2,2)}  & x_{2(1,3)}  & x_{2(2,3)} & x_{2(1,4)} & x_{2(2,4)} & x_{2(1,5)} & x_{2(1,6)} \\
    x_{3(2,1)}  & x_{3(2,2)}  & x_{3(1,3)}  & x_{3(2,3)} & x_{3(1,4)} & x_{3(2,4)} & x_{3(1,5)} & x_{3(1,6)} \\
    x_{4(2,1)}  & x_{4(2,2)}  & x_{4(1,3)}  & x_{4(2,3)} & x_{4(1,4)} & x_{4(2,4)} & x_{4(1,5)} & x_{4(1,6)} 
    \CodeAfter
    \tikz \draw[red] ([xshift=3pt, yshift=-1.5pt] 1-|3) rectangle ([xshift=-3pt, yshift=-1.5pt]5-|5);
    \tikz \draw[red] ([xshift=3pt, yshift=-1.5pt]1-|5) rectangle ([xshift=-3pt, yshift=-1.5pt]5-|7);
    \tikz \draw[blue] ([yshift=1.5pt]1-|3) rectangle ([yshift=-4.5pt]5-|7);
    \tikz \draw[orange] ([yshift=4pt]1-|1) rectangle ([xshift=3pt ,yshift=-7pt]5-|7);
    \tikz \draw[orange] ([xshift=-3pt, yshift=7pt]1-|3) rectangle ([xshift = -3pt, yshift=-10pt]5-|9);
\end{pNiceMatrix}.
\end{equation*}\\
The ideal $J_{S'}$ is generated by $2$-minors in the red rectangles,  $3$-minors in the blue rectangle, and $4$-minors in the orange rectangles. 
This generating set is a Gr\"{o}bner basis with respect to our term order.
Note that the ideal $J_{S'}$ is generated by minors of submatrices consisting of adjacent columns.
\end{example}
For a fixed combinatorial type $(u,v) \neq (0,0)$, we choose the $u$ elements in the first row slice of $[2] \times [k_2]$ to be left-justified and the $v$ elements in the second row slice of $[2] \times [k_2]$ to be right-justified. 
In symbols, we choose $S = \{1\} \times \{1, \ldots, u\} \cup \{2\} \times \{k_2 - v + 1, \ldots, k_2\}$. Then, the ideal $J_S$ is defined as
\begin{align*}
    J_S \coloneqq I_t(\hat{X}[(2,1), (2,k_2 - v)]) &+ I_t(\hat{X}[(1,u+1)), (1,k_2)])\\ \nonumber
    &+ I_{t-1}(\hat{X}[(1,u+1), (2, k_2 - v)])
        + \sum_{i = u+1}^{ k_2 - v} I_2(\hat{X}[(1,i),(2,i)]),
\end{align*}
where $\hat{X}[(a,b), (c,d)]$ denotes the submatrix of $\hat{X}$ obtained by excluding the columns indexed by $S$,~i.e., $\hat{X}_{\{(a,b),\ldots,(c,d)\} \backslash S}$.
Crucially, this choice of representative $S$ guarantees that $J_S$ is the sum of ideals generated by minors from submatrices indexed by \emph{adjacent} columns of $\hat X$. We can then apply the following result. 
\begin{theorem}[Corollary 2.4, \cite{seccia}]\label{thm:minorsconsec}
    Let $I$ be an ideal of the form
    \[I=I_{t_1}(\hat{X}[(a_1,b_1),(c_1,d_1)])+\cdots+I_{t_r}(\hat{X}[(a_r,b_r),(c_r,d_r)]).\]
    Then
    \[F_{t_1}(\hat{X}[(a_1,b_1),(c_1,d_1)])\cup\cdots\cup F_{t_r}(\hat{X}[(a_r,b_r),(c_r,d_r)])\]
    forms a Gr\"obner basis for the ideal $I$.
\end{theorem}

Note that these ideals differ from the ideals of adjacent minors (and their generalizations) studied in \cite{HS04, mohammadi2018prime}, though they include them as special cases.

As immediate corollaries of Theorem~\ref{thm:minorsconsec} we have the following:

\begin{corollary}\label{prop:Iemptygb}
Suppose $2\leq k_1$ and  $2 \leq t \leq \min\{ k_2,d\}$.
The set $F_\emptyset$ is a Gr\"{o}bner basis for~$I_\emptyset$.
\end{corollary}
\begin{proof}
Apply Theorem~\ref{thm:minorsconsec} to \eqref{eq:Iemptygens}.
\end{proof}
\begin{corollary}\label{prop:ISgb}
Suppose $k_1 = 2$ and $2 \leq t \leq \min\{ k_2,d\}$. 
If $\emptyset \neq S \subseteq[2 k_2]$ is minimal of combinatorial type $(u, v)$, where $S = \{1\} \times \{1, \ldots, u\} \cup \{2\} \times \{k_2 - v + 1, \ldots, k_2\}$, then the 
set $F(J_S)$ in \eqref{eq:FJS} forms a Gr\"{o}bner basis for $J_S$.  
\end{corollary}
\begin{proof} 
    For this particular representative $S$, $F(J_S)$ equals
    \begin{align}\label{eq:JSgens}
        F(J_S)&=F_t(\hat{X}[(2,1), (2,k_2 - v)]) \cup F_t(\hat{X}[(1,u+1), (1,k_2)]) \\
        &\cup F_{t-1}(\hat{X}[(1,u+1), (2,k_2 - v)])
    \cup \bigcup_{i = u+1}^{ k_2 - v} F_2(\hat{X}[(1,i),(2,i)])
    \end{align}
    and we may apply Theorem~\ref{thm:minorsconsec}.
\end{proof}

By Remark~\ref{rem:J_S}, since \eqref{eq:JSgens} does not involve variables from the columns indexed by $S$, it can be extended to a Gr\"{o}bner basis for $I_S$ by adding the variables from $F_1(X_S)$. 
Since the leading terms of the Gr\"{o}bner bases of $J_S$ and $I_S$ are squarefree, we have the immediate corollary:
\begin{corollary} \label{cor:ISradical}
Suppose $2 \leq t \leq \min\{ k_2,d\}$.
    For $k\geq 2$, the ideal $I_\emptyset$ is radical. For $k_1=2$, the ideals $I_S$ and $J_S$ are radical for all minimal $S$.
\end{corollary}

We now present a parametrization of the variety $V_S = V(J_S)$ when $d = t$ and $k_1=2$, which shows that $V(J_S)$ is irreducible, and hence $J_S$ is prime. This result also holds for all $d \geq t$ and $k\geq 2$ when $S = \emptyset$, which we will address separately. The key idea behind the parametrization is to first find a basis for the column space of each submatrix from which we take minors, and then express the columns of these submatrices in terms of the chosen basis using appropriate coefficients.

\begin{theorem}\label{thm:param-empty}
Suppose $2\leq k_1$ and $2\leq t \leq \min\{ k_2,d\}$.
    Then the ideal $I_\emptyset$ is prime.
\end{theorem}

\begin{proof}
 The ideal $I_\emptyset$ is radical by Corollary~\ref{cor:ISradical}. To show that it is prime, we define the space
$$
X = \mathbb{C}^{d \times(t-1)} \times \mathbb{C}^{(t - 1) \times  k_2} \times \mathbb{C}^{ k_2 \times (k_1-1)},
$$

Note that $X$ is irreducible, as it is a product of irreducible varieties over the algebraically closed field $\mathbb{C}$.
We now define the map
\[
\varphi \colon X \to \mathbb{C}^{d \times k_1k_2},
\quad\text{given by}\quad (M, N, A) \mapsto M \cdot N \cdot D_A,\]
where, for each $A = (a_{ij}) \in \mathbb{C}^{ k_2 \times (k_1-1)}$, the matrix $D_A$ is the $ k_2 \times k_1k_2$ matrix defined by
\begin{eqnarray}\label{eq:D_alpha}
    D_A = \left (
\begin{array}{ccccccccccccccccccccccccccc}
1 & a_{11} & \cdots &a_{1,k_1-1} \\
& & & & 1 & a_{21} & \cdots & a_{2,k_1-1} \\
& & &  & & & & & \ddots \\
& & &  &  &  & & & & 1 & a_{ k_21} & \cdots & a_{ k_2,k_1-1}\\
\end{array}\right ).
\end{eqnarray}

By construction, the coordinates of the points in the image of $\varphi$ are polynomials in the entries of $M,N,A$. Therefore, $\varphi$ is a polynomial map, and thus continuous. Hence, the image of $\varphi$ is irreducible. We now proceed to show that ${\rm im}(\varphi)=V_\emptyset$.

 It is clear that ${\rm im}(\varphi) \subseteq V_\emptyset$. This follows from the fact that $M$ has rank at most $t - 1$, so the product $\varphi(M, N, A)$ also has rank at most $t - 1$. Additionally, multiplication by the matrix $D_A$ ensures that specific pairs of columns in $\varphi(M, N, A)$ are properly identified.

  To show that $V_\emptyset \subseteq {\rm im}(\varphi)$, let $C \in V_\emptyset$. One obtains the $a_{ij}$ for $i \in [ k_2]$ and $j \in [k_1-1]$ by taking the ratios between columns $(1,i)$ and $(1+j,i)$ of $C$. Let the columns of $M$ form any spanning set for the column space of $C$. Finally, define the columns of $N$ to be the coefficients needed to express column $(1,i)$ of $C$ in terms of the spanning set given by $M$ for $i \in [k_2]$. 
  
Thus, we have $V_\emptyset = {\rm im}(\varphi)$, and since ${\rm im}(\varphi)$ is irreducible, the proof is complete.
\end{proof}

\begin{theorem}\label{thm:ISprime}
    Suppose $k_1=2\leq t \leq \min\{ k_2,d\}$.
    Then the ideal $I_S + I_{t + 1}(\hat X)$ is prime for all minimal $S \neq \emptyset$.
    In particular, $I_S$ is prime when $d = t$.
\end{theorem}

\begin{proof}
    Suppose $S$ is nonempty and has combinatorial type $(u,v)$. 
    Note that $V(I_S + I_{t + 1}(\hat X)) \cong V(J_S + I_{t + 1}(\hat X))$ via removing zero columns from $X$ to obtain $\hat{X}$.
    By symmetry, it is sufficient to prove that $J_S + I_{t + 1}(\hat X)$ is prime when $S = \{(1,1), (1,2), \ldots, (1,u)\} \cup \{(2, k_2 - v + 1), \ldots, (2,k_2-1), (2, k_2)\}$.

We follow the same strategy as in Theorem~\ref{thm:param-empty} and parameterize the variety $V(J_S + I_{t + 1}(\hat X))$ as follows:
    \begin{align*}
        &\varphi \colon \CC^{d \times t} \, \times \, \CC^{(t - 1) \times u} \, \times \, (\CC^{(t - 2) \times ( k_2 - u - v)}\, \times \, \CC^{ k_2-u-v}) \, \times\, \CC^{(t - 1) \times v}  \to \CC^{d\times (2 k_2-u-v)}\\
        &(M, N_1, N_2, A, N_3)  \mapsto 
        \begin{pmatrix}
        M[1, t-1]\cdot
        N_1 & &
        M[2, t-1] \cdot N_2  \cdot D_A 
        & &
        M[2, t] \cdot
        N_3
        \end{pmatrix}
    \end{align*}
    where the matrix $D_A$ is as in \eqref{eq:D_alpha} and $M[a,b]$ denotes the submatrix of $M$ with columns $\{a,\ldots,b\}$.
    Here $A$ has dimensions $( k_2 - u - v) \times 1$ so $D_A$ has dimensions $( k_2 - u - v) \times 2( k_2 - u- v)$.
    The image is the horizontal concatenation of three matrices.

    It is clear that ${\rm im}(\varphi) \subseteq V(J_S + I_{t + 1}(\hat X))$. This follows from the following points:
    \begin{itemize}
  \item Since $M[1, t-1]$ and $M[2, t]$ have rank at most $t - 1$, the products $M[1, t - 1] \cdot N_1$ and $M[2, t] \cdot N_3$ each have rank at most $t - 1$.
  \item Since ${\rm rank}(M[2, t - 1]) \leq t - 2$, the product $M[2, t - 1] \cdot N_2 \cdot D_A$ has rank at most $t - 2$.
  \item Consecutive pairs of columns in the product $M[2, t - 1] \cdot N_2 \cdot D_A$ are linearly dependent due to the structure of $D_A$.
  \item The full matrix has rank at most $t$ since all its columns lie in the column span of $M$.
\end{itemize}
    We now show that the map $\varphi$ is surjective onto $V(J_S + I_{t + 1}(\hat X))$. Let $C \in V(J_S + I_{t + 1}(\hat X))$. We describe how to construct a tuple $(M, N_1, N_2,A, N_3)$ such that $\varphi(M, N_1, N_2, A, N_3) = C$.

    To determine the values $a_i$ for $i = 1, \ldots,  k_2 - u - v$, take the ratio of columns $(1,u+i)$ and $(2,u+i)$ of $C$ (the indexing is inherited from the matrix with zero columns).  
Let $C_1$, $C_2$, and $C_3$ denote the column spaces of the submatrices $C[(1,1), (2,k_2 - v)]$, $C[(1,u+1), (2,k_2 - v)]$, and $C[(1,u + 1), (2, k_2)]$, respectively.  
Observe that $C_2 \subseteq C_1 \cap C_3$.  
Let $\mathcal{B}_2$ be a basis of $C_1 \cap C_3$, and let $\mathcal{B}_1$ and $\mathcal{B}_3$ be bases for the orthogonal complements of $C_1 \cap C_3$ in $C_1$ and $C_3$, respectively.  
Note that the total number of basis vectors satisfies $|\mathcal{B}_1 \cup \mathcal{B}_2 \cup \mathcal{B}_3| \leq t$, and any of the sets $\mathcal{B}_i$ may be empty.

Set the first column of $M$, denoted $M[1]$, to be an element of $\mathcal{B}_1$, and the last column, $M[t]$, to be an element of $\mathcal{B}_3$.  
Let the remaining $t - 2$ middle columns of $M$ be chosen from the set $\mathcal{B}_1 \cup \mathcal{B}_2 \cup \mathcal{B}_3 \setminus \{M[1], M[t]\}$.
If any of the sets $\mathcal{B}_1$, $\mathcal{B}_3$, or $\mathcal B_1 \cup \mathcal B_2 \cup \mathcal B_3 \backslash \{M[1], M[t]\}$ is empty or does not contain enough elements, the remaining columns of $M$ can be chosen arbitrarily.  
Once $M$ is constructed, we can find the matrices $N_1, N_2, N_3$ with the appropriate coefficients to obtain~$C$.

The equality ${\rm im}(\varphi) = V(J_S + I_{t+1}(\hat{X}))$, together with the fact that ${\rm im}(\varphi)$ is the image of an irreducible variety under a polynomial map and is irreducible itself, completes the proof.
\end{proof}

We conclude this section with the following example, which shows that the ideal $I_S$ for $S\neq\emptyset$ is not necessarily prime when $d>t$.
\begin{example}\label{ex:bigger_d}
Recall the second ideal in Example~\ref{example_Sec4}, where $t = 4$, $ k_2 = 6$, and 
\[S = \{(1,1), (1,2), (2,5), (2,6)\}.\] 
By Theorem~\ref{thm:set_min_primes}, the ideal $I_S$ is minimal among the set of ideals $\{I_S : S \subseteq [k_1]\times[k_2]\}$.
If $d = 4$, the ideal is prime by Theorem~\ref{thm:ISprime}. 
But if $d = 5$, the ideal $I_S$ decomposes into two prime components:
\[ I(\bar H_1) = I_S + I_2(X_{\{1,2\} \times \{3,4\}})  \quad {\rm and} \quad  I(\bar H_2) = I_S + I_5(\hat{X}),\]
where 
\begin{align*}
    H_1 &\coloneqq H(S) \cup \{\text{2-subsets of } \{1,2\} \times \{3,4\}\}\\
    H_2 &\coloneqq H(S) \cup \{\text{all 5-subsets of } [2]\times[6]\}.
\end{align*}
The first ideal $ I(\bar H_1)$ captures the situation where the columns of $X_{\{(1,3),(1,4)\}}$ are dependent, i.e., the matrix $X_{\{1,2\} \times \{3,4\}}$ has rank $\leq 1$. 
Theorem~\ref{thm:ISprime} proves that the second ideal $I(\bar H_2)$ is prime; this ideal captures the situation where the columns of $X_{\{(1,3),(1,4)\}}$ are independent, and therefore $X$ has rank $\leq 4$.
\end{example}

\section{Dimension}\label{sec:dim}
In this section, we analyze the dimensions of $I_\emptyset$ for arbitrary $k_1$ and $d\geq t$, and $I_S$ for $S\neq \emptyset$ in the case when $k_1=2$ and $d=t$. We use the parametrizations given in Theorems \ref{thm:param-empty} and \ref{thm:ISprime}. We study the dimensions of generic fibers by exhibiting symmetries in the parameter space.

\begin{theorem}\label{thm:dim-empty}
    For $2\leq k_1$ and $2 \leq t \leq \min\{ k_2,d\}$, the dimension of $I_\emptyset$ is $(t-1)(d+ k_2) +  k_2(k_1-1) - (t-1)^2 = (t-1)(d+ k_2-t+1) +  k_2(k_1-1)$.
\end{theorem}
\begin{proof}
The expected dimension of $I_\emptyset$ is $(t-1)(d+ k_2) +  k_2(k_1-1)$, obtained by counting parameters in the map 
$$\varphi: \CC^{d\times (t-1)}\times \CC^{(t-1)\times  k_2}\times \CC^{ k_2\times (k_1-1)} \to \CC^{d\times k_1k_2} \quad \text{with}\quad(M,N,A)\mapsto M\cdot N\cdot D_A.$$
However, the general linear group $\mathrm{GL}_{t-1}$ acts on the parameters: multiplying $M$ on the right by an invertible $(t-1)\times(t-1)$ matrix $C$ (yielding $MC$) and $N$ on the left by $C^{-1}$ (yielding $C^{-1}N$) leaves the image unchanged. We will show that this is the only symmetry.\medskip

We proceed similarly to \cite[Lemma 3.1]{henry2025geometry}. Suppose that we have 
\begin{align} \label{eq: generic fiber}
    MND_A = M'N'D_{A'}
\end{align}
for two different choices of parameters. We will show that $A = A'$,
$M' = MC$, and $N' = C^{-1}N$ 
for a unique $C\in\text{GL}_{t-1}$.
First note that by rearranging the columns of $D_A$ and $D_{A'}$, these matrices can be viewed as
\begin{align*}
    &D_A = \begin{pmatrix}
    I_{k_2} & {\rm diag} (a_{11},\ldots,a_{ k_21}) & \cdots & {\rm diag} (a_{1,k_1-1},\ldots,a_{ k_2,k_1-1})
    \end{pmatrix}, \\
    &D_{A'} = \begin{pmatrix}
    I_{k_2} & {\rm diag} (a_{11}',\ldots,a_{ k_21}') & \cdots & {\rm diag} (a_{1,k_1-1}',\ldots,a_{ k_2,k_1-1}')
    \end{pmatrix}.
\end{align*}
Therefore by \eqref{eq: generic fiber}, $M N = M'N' = : P $ and $a_{ij} P_i = a_{ij}' P_i$ for each column $P_i$ of $P$ and each $j\in [k_1-1]$. Thus, $A=A'$. Note that generically $M$ and $N$ have rank $t-1$. So, we can re-write 
$$M = \begin{pmatrix}
    M_1\\
    M_2
\end{pmatrix},\quad M' = \begin{pmatrix}
    M'_1\\
    M'_2
\end{pmatrix},\quad
N = \begin{pmatrix}
    N_1 &
    N_2
\end{pmatrix},\quad N' = \begin{pmatrix}
    N'_1 &
    N'_2
\end{pmatrix},$$
where $M_1, M_1', N_1, N_1'\in \text{GL}_{t-1}$, $M_2, M_2'\in \CC^{(d-t+1)\times (t-1)}$, and $N_2, N_2'\in \CC^{(t-1)\times ( k_2 - t +1)}$. Now there is a unique $C\in\text{GL}_{t-1}$ such that $M_1C = M_1'$. By assumption, we also have $$\begin{pmatrix}
    M_1N_1 & M_1N_2\\
    M_2N_1 & M_2N_2
\end{pmatrix} = MN = M'N' = \begin{pmatrix}
    M_1'N_1' & M_1'N_2'\\
    M_2'N_1' & M_2'N_2'
\end{pmatrix},$$
which implies $C^{-1}N_1 = N_1'$, $C^{-1}N_2 = N_2'$, and $M_2C = M_2'$. Therefore, $MC = M'$ and $C^{-1}N = N'$, as desired. Hence, the generic fibers of $\varphi$ are $(t-1)^2$-dimensional, and the conclusion follows.
\end{proof}

\begin{proposition}
\label{thm:dim-non-empty-bound}
    If $k_1=2\leq d = t \leq  k_2$, and $S\neq\emptyset$, then the dimension of $I_S$ is bounded above by $t^2 + (t-1) k_2 - (t-1)^2 - 1$.
\end{proposition}
\begin{proof}
    The expected dimension of $I_S$ 
    is $t^2+ (t-1) k_2$, obtained by counting parameters in the map 
    \begin{align*}
        &\varphi \colon \CC^{t \times t} \, \times \, \CC^{(t - 1) \times u} \, \times \, (\CC^{(t - 2) \times ( k_2 - u - v)}\, \times \, \CC^{ k_2-u-v}) \, \times\, \CC^{(t - 1) \times v}  \to \CC^{t\times (2 k_2-u-v)}\\
        &(M, N_1, N_2, A, N_3)  \mapsto 
        \begin{pmatrix}
        M[1, t-1]\cdot
        N_1 & &
        M[2, t-1] \cdot N_2  \cdot D_A 
        & &
        M[2, t] \cdot
        N_3 
        \end{pmatrix}.
    \end{align*}
We show that for any choice of parameters $(M,N_1,N_2,A,N_3)$, we have 
\begin{align} \label{eq: fiber of I_S}
    \varphi(M,N_1,N_2,A,N_3) = \varphi(MC,\  \begin{pmatrix} \lambda & 0 \\ b_1 & B \end{pmatrix}^{-1} N_1,\  B^{-1}N_2, \ A, \  \begin{pmatrix}B & b_2 \\ 0 & \mu \end{pmatrix}^{-1}N_3),
\end{align}
where $$C = \begin{pmatrix}
    \lambda & 0 & 0 \\
    b_1 & B & b_2 \\
    0 & 0 & \mu
\end{pmatrix},\quad B\in \text{GL}_{t-2},\quad b_1, b_2 \in\CC^{t-2}, \quad \lambda, \mu\in\CC.$$
Let $M = \begin{pmatrix} m_1 & M_2 & m_3 \end{pmatrix}$, where $m_1,m_3 \in \mathbb{C}^{t}$ and $M_2 \in \mathbb{C}^{t \times (t-2)}$. Then 
\begin{align*}
    MC = \begin{pmatrix}\lambda m_1 + M_2 b_1 & M_2 B & M_2 b_2 + \mu m_3 \end{pmatrix}.
\end{align*}
Equation~\eqref{eq: fiber of I_S} follows from observing that
\begin{align*}
    (MC)[1,t-1] = \begin{pmatrix}m_1 & M_2 \end{pmatrix} \begin{pmatrix} \lambda & 0 \\ b_1 & B \end{pmatrix}, \quad 
    (MC)[2,t] = \begin{pmatrix} M_2 & m_3\end{pmatrix}\begin{pmatrix} B & b_2 \\ 0 & \mu \end{pmatrix}.
\end{align*}

This shows that the generic fiber of $\varphi$ has dimension at least $(t-2)^2 + 2(t-2) + 2= (t-1)^2 + 1$.
\end{proof}
\begin{theorem}\label{thm:dim-non-empty}
    If $k_1=2\leq d = t \leq  k_2$, and the ideal $I_S$ is minimal with $S\neq\emptyset$, then the dimension of $I_S$ is $t^2 + (t-1) k_2 - (t-1)^2 - 1$, i.e., the bound in Proposition~\ref{thm:dim-non-empty-bound} is tight.
\end{theorem}
\begin{proof}
    Consider the map $\varphi$ from Proposition~\ref{thm:dim-non-empty-bound}. We will show that the choice of $C$ in the proof of that result is unique. Assume $\varphi(M, N_1, N_2, A, N_3) = \varphi(M', N'_1, N'_2, A', N'_3)$ is a generic point in the image of $\varphi$. Since $M$ generically has rank $t$, we can relate $M' = MC$ for some unique $C\in \text{GL}_{t-2}$ where 
$$C = \begin{pmatrix}
    \lambda & a_1 & \nu \\
    b_1 & B & b_2 \\
    \omega & a_2 & \mu
\end{pmatrix}, \quad B \in \mathbb{C}^{(t-2)\times(t-2)},\quad b_1, b_2 \in \mathbb{C}^{t-2}, \quad a_1, a_2 \in \mathbb{C}^{1\times (t-2)}, \quad\lambda, \nu, \omega, \mu \in \mathbb{C}.
$$ 
Assume $M = \begin{pmatrix} m_1 & M_2 & m_3\end{pmatrix}$ and $M' = \begin{pmatrix} m_1' & M_2' & m_3'\end{pmatrix}$, where $m_1,m_3,m_1',m_3' \in \mathbb{C}^{t}$ and $M_2,M_2' \in \mathbb{C}^{t\times (t-2)}$. Then 
\begin{align} 
\label{eq: M'}
     M' = MC = \begin{pmatrix}
    \lambda m_1 + M_2 b_1 + \omega m_3  & m_1 a_1 + M_2 B + m_3 a_2 & \nu m_1  + M_2 b_2 + \mu m_3
\end{pmatrix}.
\end{align}
Now without loss of generality re-write $$\begin{pmatrix} m_1 & M_2 \end{pmatrix} = \begin{pmatrix} P \\ p \end{pmatrix} \text{ and } \begin{pmatrix} m_1' & M_2' \end{pmatrix} = \begin{pmatrix} P' \\ p' \end{pmatrix}, \text{ where } P, P'\in \text{GL}_{t-1}, \text{ and } p, p'\in\CC^{1\times(t-1)}.$$ Then there exists a unique matrix $G\in {\rm GL}_{t-1}$ such that $PG=P'$. 
Now, define
\begin{align*}
    N \coloneqq \begin{pmatrix} N_1 & \begin{array}{c} 0_{1 \times ( k_2-u-v)} \\ N_2 \end{array} \end{pmatrix}, \quad
    N' \coloneqq \begin{pmatrix} N_1' & \begin{array}{c} 0_{1 \times ( k_2-u-v)} \\ N_2' \end{array} \end{pmatrix}.
\end{align*}
Since $\begin{pmatrix} m_1 & M_2 \end{pmatrix} N = \begin{pmatrix} m_1' & M_2' \end{pmatrix} N'$, we have $PN = P'N'$ and $pN = p'N'$. From the first equality we get $N = G N'$. So, from the second equality, $p G N' =p' N'$. Since $S$ is minimal, we have $t-1 \leq  k_2-v$, and so the matrix $N'$ has a right inverse. This means that $p G  = p'$, therefore 
$\begin{pmatrix} m_1 & M_2 \end{pmatrix} G = \begin{pmatrix} m_1' & M_2' \end{pmatrix}$.
Thus, by \eqref{eq: M'},
\begin{align}\label{eq:m1M2G}
\begin{pmatrix} m_1 & M_2 \end{pmatrix} G = \begin{pmatrix}
    \lambda m_1 + M_2b_1 + \omega m_3  & m_1 a_1 + M_2 B + m_3 a_2
\end{pmatrix}.
\end{align}
The columns on the left-hand side of \eqref{eq:m1M2G} are linear combinations of $m_1$ and the columns of $M_2$, while the columns on the right-hand side of \eqref{eq:m1M2G} are linear combinations of $m_1$, the columns of $M_2$, and $m_3$, which proves that $\omega = 0$ and $a_2 = 0$. Similarly, one can prove that $\nu=0$ and $a_1 = 0$ by considering the matrices $\begin{pmatrix}
    M_2 & m_3
\end{pmatrix}$ and $\begin{pmatrix}
    M_2' & m_3'
\end{pmatrix}$. 
\end{proof}

\begin{example} \label{ex:dim-and-degree}
For the case $k_1 = 2$, $ k_2 = 6$, and $d=t = 4$, the degrees and dimensions corresponding to the different minimal combinatorial types are as follows.
    \begin{center}
    \small\begin{tabular}{c|c|c|c|c}
        Combinatorial type & Ideal representative & Number of such ideals & Dimension & Degree  \\
        \hline
        $(0,0)$ & $I_\emptyset$ & 1 & 27 & 34560 \\
        $(1,1)$ & $I_{\{(1,1), (2,2)\}}$ & 30 &  24 & 1410 \\
        $(1,2)$  & $I_{\{(1,1),(2,2),(2,3)\}}$ &  120 &  24 & 606 \\
        $(1,3)$  & $I_{\{(1,1), (2,2), (2,3), (2,4)\}}$ & 120 & 24 & 129 \\
        $(2,2)$ & $I_{\{(1,1), (1,2), (2,3), (2,4)\}}$ & 90 & 24 & 194 \\
        $(2,3)$  & $I_{\{(1,1), (1,2), (2,3), (2,4), (2,5)\}}$ & 120 & 24 & 15\\
        $(3,3)$ & $I_{\{(1,1), (1,2), (1,3), (2,4), (2,5), (2,6)\}}$ & 20 & 24 & 1 \\
    \end{tabular}
\end{center}

Note that for non-minimal combinatorial types, the bound proved in Proposition~\ref{thm:dim-non-empty-bound} may no longer be tight. For example, for the combinatorial type $(0, 6)$, giving rise to prime but non-minimal ideals $I_S$, the dimension is 21.\footnote{\url{https://github.com/yuliaalexandr/decomposing-conditional-independence-ideals-with-hidden-variables}}
\end{example}

\noindent{\bf Future work.} 
In this paper, we provided a minimal prime decomposition of the ideal $\mathcal{I}_C$ in the case $k_1 = 2$ and $d = t$. There are several natural extensions of the results presented here.

First, in Theorems~\ref{thm:dim-empty} and~\ref{thm:dim-non-empty}, we established the dimensions of the ideal $I_\emptyset$ and the minimal ideals $I_S$ (for $S \neq \emptyset$, with $k_1 = 2$ and $d = t$). The formulas show that all ideals $I_S$ with $S \neq \emptyset$ have the same dimension as long as $I_S$ is minimal. A natural next step is to investigate their degrees and seek exact formulas. We present a computation of these degrees in Example~\ref{ex:dim-and-degree}.

Second, while the decomposition given in Theorem~\ref{thm: decomposition theorem} applies for general values of $k$ and $d$, the ideals $I_S$ may not be prime in general, as illustrated in Examples~\ref{ex:bigger_d} and~\ref{ex:bigger_k}. Hence, another natural direction is to determine the full primary decomposition of the ideals $I_S$ and $\mathcal{I}_C$ in the general setting.

Third,  matroids arising from grids are studied in~\cite{liwski2024paving}. Their connection to the minimal primes of $\mathcal{I}_C$ may yield combinatorial insight into the structure of these ideals.

\medskip 
\noindent
\textbf{Acknowledgments.}~We thank Aldo Conca for directing us to \cite{seccia}, and Emiliano Liwski for sharing the parametrization presented in Theorem~\ref{thm:ISprime}. This project originated at the Workshop for Women in Algebraic Statistics, held at St John's College, Oxford, from July 8 to July 18, 2024. We are grateful to Jane Ivy Coons for organizing the workshop and creating a collaborative environment that made this work possible. F.M. gratefully acknowledges the hospitality of the Mathematics Department at Stockholm University during her research visit, where part of this work was carried out, and thanks Samuel Lundqvist for his support during the visit. 

\medskip
\noindent
\textbf{Funding.}~The Workshop for Women in Algebraic Statistics was supported by St John's College, Oxford, the L'Oreal-UNESCO for Women in Science UK and Ireland Rising Talent Award in Mathematics and Computer Science (awarded to Jane Coons),
and the UKRI/EPSRC Additional Funding Programme for the Mathematical Sciences. F.M. was partially supported by the FWO grants G0F5921N (Odysseus) and G023721N, and the KU Leuven grant iBOF/23/064.  P.S. was supported by the Vanier Canada Graduate Scholarship, NSERC Discovery Grant DGECR-2020-00338, and the Canada CIFAR AI Chair grant awarded to Elina Robeva. T.Y. was supported by NSF grants DGE-2241144 and DMS-1840234.

\bibliographystyle{abbrv}
\bibliography{revision-oct-14.bib}

\medskip
{\small\noindent {\bf Authors' addresses}
\medskip

\noindent{Yulia Alexandr,  UCLA}\hfill {\tt yulia@math.ucla.edu}
\\ 
\noindent{Kristen Dawson, San Francisco State University} \hfill {\tt kdawson1@sfsu.edu}
\\
\noindent{Hannah Friedman, UC Berkeley} \hfill {\tt hannahfriedman@berkeley.edu}
\\
\noindent{Fatemeh Mohammadi, 
KU Leuven} \hfill {\tt fatemeh.mohammadi@kuleuven.be}
\\ 
\noindent{Pardis Semnani, University of British Columbia}\hfill {\tt  psemnani@math.ubc.ca} 
\\ 
\noindent{Teresa Yu, University of Michigan} \hfill {\tt twyu@umich.edu}
}\\

\end{document}